\documentclass[11pt, reqno]{article}
\usepackage{amsfonts}
\usepackage{amsmath}
\usepackage{amsthm}
\usepackage{amssymb}
\usepackage{lastpage}
\usepackage{fancyhdr}
\usepackage{fancybox}
\usepackage{mathrsfs}

\def\phi{\varphi }
\def\epsilon{\varepsilon}

\swapnumbers

\def\arcosh{{\rm arcosh}\>}
\theoremstyle{plain}
\newtheorem{theorem}{Theorem}[section]

\newtheorem{lemma}[theorem]{Lemma}
\newtheorem{proposition}[theorem]{Proposition}
\theoremstyle{definition}
\newtheorem{definition}[theorem]{Definition}
\newtheorem{remark}[theorem]{Remark}

\newtheorem{example}[theorem]{Example}

\numberwithin{equation}{section}

\begin{document}

\title{ Product formulas for a two-parameter family of Heckman-Opdam
hypergeometric functions of type BC}
\author{
Michael Voit\\
Fakult\"at Mathematik, Technische Universit\"at Dortmund\\
          Vogelpothsweg 87,
          D-44221 Dortmund, Germany\\
e-mail:  michael.voit@math.tu-dortmund.de}
\date{\today}

\maketitle

\begin{abstract}
In this paper we present explicit product formulas for a continuous
two-parameter family of Heckman-Opdam hypergeometric functions of type $BC$ on
  Weyl chambers $C_q\subset \mathbb R^q$ of type $B$. 
These formulas are related to  continuous one-parameter families of 
probability-preserving convolution
structures on $C_q\times\mathbb R$. 
These convolutions on $C_q\times\mathbb R$ are constructed via  product
formulas for the spherical functions of the
 symmetric spaces $U(p,q)/ (U(p)\times SU(q))$ and associated  
  double coset convolutions on $C_q\times\mathbb T$ with the
torus $\mathbb T$.
We shall obtain positive
  product formulas 
for a restricted parameter set only, while the associated convolutions are
always norm-decreasing.

Our paper is related to recent positive 
product formulas of R\"osler for three
series of Heckman-Opdam hypergeometric functions of type $BC$ as well as 
to classical product formulas for Jacobi functions of Koornwinder and Trimeche 
for  rank $q=1$. 
\end{abstract}




\section{Introduction}
 It is well-known by the work of Heckman and Opdam
 (\cite{H}, \cite{HS}, \cite{O1}, \cite{O2}) that the spherical functions on
 the Grassmann manifolds of rank $q\ge1$ over the fields 
$\mathbb F=\mathbb R, \mathbb C, \mathbb H$ may be regarded as 
  Heckman-Opdam hypergeometric functions $F_{BC_q}$ of type BC on the closed 
Weyl chambers 
$$C_q:=\{t=(t_1,\cdots,t_q)\in\mathbb  R^q: \> t_1\ge t_2\ge\cdots\ge t_q\ge0\}$$
of type $B$. The associated product formulas for the spherical functions were
stated by R\"osler \cite{R2} for the dimensions $p\ge 2q$ in a form such that
these formulas can be extended by some principle of analytic continuation to
all parameters $p\in\mathbb R$ with $p> 2q-1$. In this way,  R\"osler
\cite{R2} obtained three continuous series of product formulas for  $F_{BC_q}$
(for $\mathbb F=\mathbb R, \mathbb C, \mathbb H$) as well as associated
commutative, probability-preserving convolution algebras of measures on $C_q$, so-called
commutative hypergroups. For the theory of hypergroups we refer to \cite{J}
and \cite{BH}.

In this paper we start with  the symmetric spaces $U(p,q)/(U(p)\times
SU(q))$ for $\mathbb F=\mathbb C$ and $p\ge 2q$. Here the spherical functions can be
regarded as functions on $C_q\times\mathbb T$ with the
torus $\mathbb T:=\{z\in\mathbb C:\> |z|=1\}$. These functions
  can  be expressed in terms of $K$-spherical functions of the Hermitian symmetric spaces
$U(p,q)/(U(p)\times U(q))$  and thus also in terms
 of the functions $F_{BC_q}$ depending on the
integer multiplicity parameter $p\ge 2q$ and some spectral parameter $l\in\mathbb Z$;
 see \cite{Sh}, \cite{HS}. Following \cite{R2}, we shall write down
 product formulas for the spherical functions of $U(p,q)/(U(p)\times SU(q))$
 in Section 2 of this paper as  product formulas on $C_q\times\mathbb T$ for
 integers  $p\ge 2q$. We then use these formulas in Section 3 to construct 
 associated product formulas on the universal covering  $C_q\times\mathbb R$ of
 $C_q\times\mathbb T$ for a class of functions which are defined in terms of 
the functions  $F_{BC_q}$ where the $F_{BC_q}$ depend now on two continuous 
multiplicity parameters $p\ge 2q-1$ and $l\in\mathbb R$. Here, the
extension from integers $p\ge 2q$ and $l\in\mathbb Z$ to real numbers  $p>
2q-1$ and $l\in\mathbb R$
is carried out by
Carleson's theorem, a  principle of analytic continuation. The degenerated 
limit case  $p= 2q-1$ then follows by continuity. We shall also see in
Section 4 that these product formulas on  $C_q\times\mathbb R$ for
$p\ge 2q-1$
lead to commutative hypergroup structures $(C_q\times\mathbb R, *_p)$.
We shall also derive the Haar measures of these hypergroups.

The product formulas on  $C_q\times\mathbb R$ and the
associated hypergroup structures in Sections 3 and 4 form the basis to derive
a lot of further product formulas and convolution algebras by taking suitable
quotients. In particular, we immediately obtain extensions of the product formulas 
on  $C_q\times\mathbb T$ for the
spherical functions of $U(p,q)/(U(p)\times SU(q))$ with integers $p\ge 2q$ to
real parameters $p\ge 2q-1$. Moreover, we  recover R\"osler's formulas in \cite{R2} on
$C_q$ for $\mathbb F=\mathbb C$. More generally, 
we obtain explicit product formulas and convolution structures on $C_q$ for
all hypergeometric functions $F_{BC_q}$ with multiplicities
$$k=(k_1,k_2,k_3)= (p-q-l,1/2+l, 1)$$
with real parameters $p\ge 2q-1$ and $l\in\mathbb R$.
It will turn out in Section 5 that these formulas lead to probability-preserving
convolutions  and thus classical commutative hypergroups for $|l|\le 1/q$
while for arbitrary $l\in\mathbb R$ the positivity of the product formulas remain
open. On the other hand we shall see in Section 6 that the product formulas for 
 $F_{BC_q}$ lead for all $l\in\mathbb R$ at least to norm-decreasing convolution
algebras which are associated with certain so-called  signed hypergroup structures on
$C_q$. For this notion we refer to \cite{R0}, \cite{RV}, and references cited there.

We also notice that for rank $q=1$, our results are closely related with the classical
work of Flensted-Jensen and  Koornwinder (\cite{F}, \cite{FK}, \cite{K}) on
Jacobi functions and to the convolutions of Trimeche \cite{T} on the space
$\{z\in\mathbb C:\> |z|\ge 1\}$ which is homeomorphic with
$[0,\infty[\times\mathbb T$. 

Before starting with the analysis of $(U(p,q)/(U(p)\times SU(q))$ in Section
2, we recapitulate some notions and facts.
For integers $p>q\ge1$
consider the Grassmann manifolds $G/K$ over  $\mathbb F=\mathbb R, \mathbb C, \mathbb H$
with 
$$G= SO_0(p,q),\, SU(p,q),\, Sp(p,q)$$
 and the maximal  compact subgroups  $$K= 
SO(p)\times SO(q), \, S(U(p)\times U(q)) , \, Sp(p)\times Sp(q),$$
respectively.
By the well-known $KAK$ decomposition of $G$, a system of representatives 
of the $K$-double cosets on $G$ is given by the matrices 
\begin{equation}\label{a-t-def}
a_t = \, \begin{pmatrix} I_{p-q}&0&0\\0 & \cosh \underline t & \sinh \underline t
 \\0 &\sinh \underline t & \cosh \underline t \end{pmatrix}
\end{equation}
with $t$ in the closed Weyl chamber 
$C_q$ where  $\cosh\underline t, \> \sinh\underline t$ are the
 $q\times q$ diagonal matrices
$$\cosh\underline t:=diag(\cosh t_1,\ldots,\cosh t_q),\>\> 
 \sinh\underline t:=diag(\sinh t_1,\ldots,\sinh t_q).$$
Therefore, continuous $K$-biinvariant functions on $G$ are in  a natural one-to-one
correspondence with continuous functions on $C_q$. Moreover, by the
theory of Heckman and Opdam \cite{H}, \cite{HS}, \cite{O1}, \cite{O2}, in
this way  the
spherical functions of $(G,K)$, i.e., the continuous, $K$-biinvariant functions
$\phi\in C(G)$ satisfying the product formula
\begin{equation}\label{prod-formula-allg}
\phi(g)\phi(h)=\int_K \phi(gkh)\> dk \quad\quad(g,h\in G,\> dk\>\> \text{the
  normalized Haar measure of}\>\> K),
\end{equation}
 are precisely  the Heckman-Opdam hypergeometric functions
$$ t\mapsto F(\lambda,k;t):=F_{BC_q}(\lambda,k;t) \quad\quad (t\in C_q)$$
of type BC with $\lambda\in\mathbb C^q$ and the multiplicity parameter
$$k=(k_1,k_2,k_3)=(d(p-q)/2, (d-1)/2, d/2)$$
associated with the roots $\pm 2 e_i$,  $\pm 4 e_i$, and  $2(\pm  e_i\pm
e_j)$ of the root system  $2\cdot BC_q$
in the notion of Heckman and Opdam. Here, $d\in\{1,2,4\}$ is  the dimension of
  $\mathbb
R,\mathbb C,\mathbb H$ over $\mathbb R$; see Remark 2.3 of \cite{H} and also
 \cite{R2}.

For given $q\ge1$ and $d$, R\"osler
 \cite{R2}  derived the  product formula
 (\ref{prod-formula-allg})
explicitly as a product formula for the corresponding $F_{BC_q}$
  on $C_q$ depending
 on  $p\ge 2q$ 
 such that this formula can be extended  to a product
 formula   on $C_q$ for 
arbitrary real parameters $p> 2q-1$ by  analytic continuation.
Moreover, for  each  real parameter $p> 2q-1$, each of these product
formulas  gives rise to a
commutative hypergroup structure on   $C_q$, i.e., a
probability preserving commutative Banach-*-algebra structure on the Banach
space of all bounded signed Borel measures on  $C_q$ with total variation norm;
 see \cite{BH} and \cite{J} for the theory of hypergroups.

We now  modify the approach of \cite{R2}  for
 $\mathbb F=\mathbb C$, i.e.  $d=2$,  by considering $K$-spherical functions 
according to Ch. I.5 of \cite{HS} and \cite{Sh} as follows: 
 Take the Gelfand pair 
$( G, \tilde K):=(U(p,q), U(p)\times SU(q))$  for $p\ge q$ as well as
the maximal compact subgroup $K:=U(p)\times U(q)\subset K$.
Then 
$$G/K:=U(p,q)/(U(p)\times U(q)) \equiv SU(p,q)/S(U(p)\times
U(q))$$ is a Hermitian symmetric space.
It is well-known that the usual  spherical functions 
of $( G, \tilde K)$ (i.e., satisfying (\ref{prod-formula-allg}))
 are
in a natural way in a  one-to-one correspondence with the so-called
 $K$-spherical functions on $G$ of type $l$ with  
$l\in\mathbb Z$. Recapitulate that these
 $K$-spherical functions of type $l$ are defined as continuous functions
 $\phi\in C(G)$
with $\phi(e)=1$ satisfying twisted invariance conditions 
as well as  twisted  product formulas associated with the characters
\begin{equation}\label{twist-character}
\chi_l\left(\begin{pmatrix} u&0\\0&v\end{pmatrix}\right):=(\Delta v)^l
\quad\quad (l\in\mathbb Z)
\end{equation}
of $K$, where $\Delta v$ stands for the determinant of $v$. The twisted invariance condition 
reads as
\begin{equation}\label{invariance-twist}
\phi\left(k_1gk_2\right)
= \chi_l(k_1k_2)^{-1}\cdot \phi(g)
\quad\text{for}\quad
g\in G ,k_1,k_2\in K,
\end{equation} 
and the  twisted  product formula as
\begin{equation}\label{prod-formula-allg-twist}
\phi(g)\phi(h)=\int_{S(U(p)\times U(q))}  \chi_l(k)   \phi(gkh)\>
 dk \quad\quad(g,h\in  G).
\end{equation}
These spherical functions $\phi$  of type $l$ can be also
characterized as eigenfunctions of some algebra $\mathbb D(\chi_l)$ of $SU(p,q)$-invariant 
differential operators in  Section 5.1 of \cite{HS} and \cite{Sh} and can be
written down therefore explicitly in terms of $F_{BC_q}$. 
In particular, if $G//\tilde K$ is identified with $C_q\times\mathbb T$,
we conclude from Theorem 5.2.2 of \cite{HS} that 
the spherical functions  may be regarded as the functions
\begin{equation}\label{def-phi-lam-l}
(t,z)\longmapsto z^l\prod_{j=1}^q cosh^{l}t_j \cdot F_{C_q}(\lambda,k(p,l);t)
\quad\quad(\lambda\in\mathbb C^q,\> l\in\mathbb Z)
\end{equation}
 with the   multiplicities
\begin{equation}\label{def-twisted-k}
k(p,l)=(k_1(p,l),k_2(p,l),k_3(p,l))=(p-q-l, 1/2+l, 1).
\end{equation}
We shall use this characterization in the next section to derive an explicit
product formula for these functions.

It is a pleasure to thank Tom Koornwinder for some essential 
 hints to details in the monograph
\cite{HS}  as well Margit R\"osler for many
fruitful discussions.

\section{Spherical functions of $(U(p,q), U(p)\times SU(q))$ and their product
  formula}

In this section we derive an explicit  product formula for the
 spherical functions for the Gelfand pair $(G,\tilde K):= (U(p,q), U(p)\times SU(q))$.
 In fact, this is a
Gelfand pair by standard criteria, see e.g. Corollary 1.5.4 of
\cite{GV}. Moreover, this is also a direct consequence of the explicit
convolution (\ref{group-torus-convo}) below. 
We first identify the double coset space $G//\tilde K$ with the direct product 
$C_q\times \mathbb T$ of the Weyl chamber $C_q$ and the torus
$\mathbb  T$.
 This can be done similar to well known 
case of the Hermitian symmetric space $G/K:=U(p,q)/(U(p)\times U(q))$ where, by
the  $KAK$-decomposition, a system of representatives of the $K$-double
cosets in $G$ is given by the matrices $a_t$ of Eq.~(\ref{a-t-def}) with $t\in
C_q$.

\begin{lemma} A system of representatives of the $\tilde K$-double
cosets in $G$ is given by the matrices
\begin{equation}\label{a-t-z-def}
a_{t,z} = \, \begin{pmatrix} I_{p-q}&0&0\\0 &  r_q(z)^{-1}\cdot \cosh \underline t & \sinh \underline t
 \\0 &\sinh \underline t &  r_q(z)\cdot\cosh \underline t \end{pmatrix}
\end{equation}
for $t\in C_q$ and $z\in\mathbb T$ where the mapping $r_q:\mathbb T\to\mathbb T
$
 is the $q$-th root on $ \mathbb T$
with $r_q(e^{it}):=e^{it/q}$ with $t\in[0,2\pi[$.
\end{lemma}

\begin{proof} We first check that each double coset has a representative of the form $a_{t,z}$.
In fact, by the well known $KAK$-decomposition of $G$, each $g\in G$ has the form
$$g=\begin{pmatrix} u_1&0\\0&v_1\end{pmatrix}
\begin{pmatrix} I_{p-q}&0&0\\0&\cosh \underline t&\sinh \underline t
\\0&\sinh \underline t&\cosh \underline t
 \end{pmatrix}
\begin{pmatrix} u_2&0\\0&v_2\end{pmatrix}$$
with $t\in C_q$, $u_1,u_2\in U(p)$, $v_1,v_2\in U(q)$. The matrices  $v_k\in U(q)$ ($k=1,2$)
can
be written as $v_k=z_k\cdot\tilde v_k$ with  $\tilde v_k\in SU(q)$ and
$z_k=r_q(\Delta(v_k))\in\mathbb T$. 
Therefore, defining $$\tilde u_1:=u_1\begin{pmatrix} I_{p-q}&0\\0&
z_2I_q\end{pmatrix}  \quad\quad\text{and}\quad\quad
\tilde u_2:=\begin{pmatrix} I_{p-q}&0\\0&
z_1I_q\end{pmatrix}u_2,$$ we obtain
\begin{equation}\label{elem-double}
g=\begin{pmatrix}\tilde  u_1&0\\0&\tilde v_1\end{pmatrix}
\begin{pmatrix} I_{p-q}&0&0\\0& (z_1z_2)^{-1}\cosh \underline t&\sinh \underline t
\\0&\sinh \underline t& z_1z_2\cosh \underline t
 \end{pmatrix}
\begin{pmatrix}\tilde u_2&0\\0&\tilde v_2\end{pmatrix},
\end{equation}
which shows that each double coset has a representative of the form 
\begin{equation}\label{elem-double-modified}
\tilde a_{t,z}:=\begin{pmatrix} I_{p-q}&0&0\\0& z^{-1}\cosh \underline t&\sinh \underline t
\\0&\sinh \underline t& z\cosh \underline t
 \end{pmatrix}
\end{equation}
with $z\in\mathbb T$. Moreover, this computation also shows that for all $q$-th roots of unity
$z_0\in\mathbb T$, the matrices $\tilde a_{t,z}$ and
 $\tilde a_{t,zz_0}$ are contained in the same $\tilde K$-double coset,
 i.e., each double coset has a representative of the form  $\tilde a_{t,z}$
 with $z=e^{i\theta}$, $\theta\in[0,2\pi/q[$, as claimed.

In order to show that the $a_{t,z}$ are contained in different double cosets
 for different $(t,z)$, we briefly discuss how the parameters $t,z$ of a double coset
of a  arbitrary group element $g\in G$ can be constructed explicitly.
 This will be important also later on for the convolution.
We write any $g\in G$ in $(p\times q)$-block notation as
$$ g = \begin{pmatrix} A(g) & B(g)\\
C(g) & D(g)     \end{pmatrix}.$$
Moreover, in $(p\times q)$-block notation, we write
$$ a_{t,z} = \,\begin{pmatrix} A_{t,z} & B_{t,z}\\
C_{t,z} & D_{t,z}     \end{pmatrix}.$$
Assume now that $g\in G$ has the form (\ref{elem-double}) with
$z_1z_2=r_q(z)$ with  $z=e^{i\theta}$, $\theta\in[0,2\pi/q[$.
Then  $D(g)$ has the form
\begin{equation}\label{Dg}
D(g)=r_q(z)\cdot\tilde v_1 \> \cosh \underline t\> \tilde v_2
\end{equation}
with $\tilde v_1,\tilde v_2\in SU(q)$.
We now consider the singular spectrum
$$\sigma_{sing}:M^{q,q}(\mathbb C)\to C_q, \quad
\sigma_{sing}(a):=\sqrt{spec(a^*a)}\in\mathbb R^q$$
where the singular values are ordered by size. We also consider the map
$arg:\mathbb C^\times\to\mathbb T$, $ arg(z):=z/|z|$.
Then, by (\ref{Dg}), 
\begin{equation}\label{ident}
t=\arcosh(\sigma_{sing}(D(g))) \quad \text{in all components} \quad\quad\text{and} \quad\quad
z=arg(\Delta(D(g))).
\end{equation}
This completes the proof of the lemma.
\end{proof}

We proceed with the notations of the second part of the proof of the lemma
 and write the general product formula (\ref{prod-formula-allg}) 
for spherical functions as a product formula on the parameter
 space $C_q\times \mathbb T$.
For this, we take $t,s\in C_q$ and $z_1,z_2\in\mathbb T$ and  evaluate the integral 
$$\int_{\tilde K} f(a_{t,z_1}ka_{s,z_2})\>dk
 \quad\quad\text{with} \quad\quad
k= \begin{pmatrix} u & 0\\
0 & v     \end{pmatrix}.$$ 
We have
$$a_{t,z_1}ka_{s,z_2}=\begin{pmatrix} * & *\\
* & C_{t,z_1}\>u\>B_{s,z_2} + D_{t,z_1}\>v\>D_{s,z_2}  \end{pmatrix}$$
and thus
$$D(a_{t,z_1}\>k\>a_{s,z_2})=C_{t,z_1}\>u\>B_{s,z_2} + D_{t,z_1}\>v\>D_{s,z_2}= 
(0,\sinh\underline t)\>u\>\begin{pmatrix}0\\ \sinh\underline s\end{pmatrix}
+r_q(z_1z_2)\>\cosh \underline t\> v\> \cosh \underline s .$$
With the block matrix 
\begin{equation}\label{def-sigma-0}
 \sigma_0 := \begin{pmatrix}0\\ I_q\end{pmatrix} \in M_{p,q}(\mathbb C)
\end{equation}
this can be written as 
$$ D(a_{t,z}\>k\>a_{s,w}) = \,\sinh\underline t \,\sigma_0^* u
\sigma_0\sinh\underline s \,+\,
r_q(z_1z_2)\>\cosh \underline t\> v\> \cosh \underline s .$$
Therefore, if we regard a $\tilde K$-biinvariant function $f\in C(G)$ as a continuous
function on $C_q\times\mathbb T$,
\begin{align}
\int_{\tilde K}& f(a_{t,z_1}ka_{s,z_2})\>dk=\notag\\
=&\int_{\tilde K} f\left(\arcosh(\sigma_{sing}(D(a_{t,z_1}\>k\>a_{s,z_2}))), 
arg(\Delta(D(a_{t,z_1}\>k\>a_{s,z_2})))\right)
\notag\\
=&
\int_{U(p)}\int_{SU(q)} f\Bigr(\arcosh(\sigma_{sing}(\sinh\underline t \,\sigma_0^* u
\sigma_0\sinh\underline s \,+\,
r_q(z_1z_2)\>\cosh \underline t\> v\> \cosh \underline s
)), \notag\\
&\quad\quad\quad\quad\quad\quad\quad\quad
arg(\Delta(   \sinh\underline t \,\sigma_0^* u
\sigma_0\sinh\underline s \,+\,
r_q(z_1z_2)\>\cosh \underline t\> v\> \cosh \underline s     ))\Bigl)\> dv\> dw.
\notag 
\end{align}
Notice that $\,\sigma_0^* u \sigma_0\in M_q(\mathbb F)$ is just
 the lower right $q\times q$-block of $\sigma$ and is contained in the ball
$$B_q:=\{w\in M^{q,q}(\mathbb C):\> w^*w\le I_q\},$$
where $w^*w\le I_q$ means that $I_q-w^*w$ is positive semidefinite. 
In order to reduce the $U(p)$-integration, 
we use  Lemma 2.1 of \cite{R2} and obtain that for $p\ge 2q$ the integral above
is equal to
\begin{align}
\frac{1}{\kappa_p}\int_{B_q}\int_{SU(q)} 
f\Bigr(\arcosh(\sigma_{sing}(&\sinh\underline t \,w\,\sinh\underline s \,+\,
r_q(z_1z_2)\>\cosh \underline t\> v\> \cosh \underline s
)),  \\
arg(\Delta(\sinh\underline t \,w\,\sinh\underline s& \,+\,
r_q(z_1z_2)\>\cosh \underline t\> v\> \cosh \underline s
))\Bigl)\cdot
\Delta(I_q-w^*w)^{p-2q}
\> dv\> dw.\notag
\end{align}
 with 
\begin{equation}
\kappa_p:=\int_{B_q}\Delta(I_q-w^*w)^{p-2q}\> dw,
\end{equation}
 and where
 $dw$ means integration w.r.t.~the Lebesgue measure. After substitution $w\mapsto
 r_q(z_1z_2)w$, we arrive at the following explicit product formula:

\begin{proposition}\label{prop-group-torus-convo}
Let $p\ge 2q$.
If a $\tilde K$-spherical function $\phi\in C(G)$ is regarded as a continuous
function on $C_q\times\mathbb T$ as described above, then the associated product
formula for spherical functions $\phi$ has the following form on  $C_q\times\mathbb T$:
\begin{align}\label{group-torus-convo}
\phi(t,z_1)\phi(s,z_2)=\quad\quad &\\
=\frac{1}{\kappa_p}\int_{B_q}\int_{SU(q)} 
\phi\Bigr(\arcosh(\sigma_{sing}(&\sinh\underline t \,w\,\sinh\underline s \,+\,
\>\cosh \underline t\> v\> \cosh \underline s
)),\notag  \\
z_1z_2\cdot arg(\Delta(&\sinh\underline t \,w\,\sinh\underline s \,+\,
\>\cosh \underline t\> v\> \cosh \underline s))\Bigl)\cdot
\Delta(I_q-w^*w)^{p-2q}
\> dv\> dw.\notag
\end{align}
\end{proposition}

Notice that for $p\le 2q-1$, the integral over the ball
 $B_q$ in (\ref{group-torus-convo}) does not exist and that here $ \kappa_p=\infty$. However, 
for $p=2q-1$, a degenerated version of formula may be written down explicitly similar 
to other series of symmetric spaces of higher rank associated with  motion groups or Heisenberg groups;
see Section 3  of \cite{R1} and Remark 2.14 of \cite{V2}. We shall present a degenerated version of 
(\ref{group-torus-convo})  in the end of Section 3.

 In the case of rank $q=1$, the space $C_q\times \mathbb T=[0,\infty[\times \mathbb T$ 
may be identified with the exterior $Z:=\{z\in\mathbb C:\> |z|\ge1\}$ of the unit disk
 via polar coordinates. The associated convolution 
(\ref{group-torus-convo}) in the general case $p\ge 2q$ and in the degenerated case $p = 2q$
was computed in this case  by Trimeche \cite{T}.

We next turn to the classification of all  $\tilde K$-spherical functions
$\phi\in C(G)$. In order to translate it into a standard version in
the Heckman-Opdam theory, we recapitulate the following result:

\begin{lemma}\label{euiv-def-spher}
 For $\phi\in C(G)$ the following properties are equivalent:
\begin{enumerate}
\item[\rm{(1)}] $\phi$ is $\tilde K$-spherical, i.e., $\tilde K$-biinvariant with 
 $\phi(g)\phi(h)=\int_{\tilde K} \phi(gkh)\> dk$ for $g,h\in G$.
\item[\rm{(2)}] $\phi(e)=1$, and there exists a unique $l\in\mathbb Z$ such that
\begin{equation}\label{twist-invar}
\phi(k_1gk_2)=\chi_l(k_1k_2)^{-1}\cdot \phi(g) \quad\quad\text{for}\quad\quad g\in G,\> k_1,k_2\in K
\end{equation}
and 
\begin{equation}\label{twist-prod}
\phi(g)\phi(h)=\int_{ K} \phi(gkh)\chi_l(k)\> dk \quad\quad\text{for}\quad\quad g,h\in G.
\end{equation}
\item[\rm{(3)}] $\phi$ is a spherical function of type $\chi_l$ for some $l\in\mathbb Z$
 in the sense of Definition 5.2.1 of Heckman \cite{HS}, i.e., $\phi$ is 
 an eigenfunction with respect to  all members of a certain algebra of 
$G$-invariant differential operators.
\end{enumerate}
\end{lemma}

\begin{proof} For abbreviation  define the diagonal matrix $d_z:= \begin{pmatrix} I_p & 0\\
0 & zI_q     \end{pmatrix}\in K$ for $z\in \mathbb T$.

For $(1)\Longrightarrow(2)$ consider a $\tilde K$-spherical function $\phi$. 
Then $\phi|_K$ is  $\tilde K$-spherical, and  as $K/\tilde K\equiv\mathbb T$, 
we find a unique $l\in\mathbb Z$ with $\phi|_K=\chi_l^{-1}$. Moreover, as 
$\phi(kg)=\phi(k)\phi(g)=\phi(gk)$ for $g\in G$ and $k\in K$, (\ref{twist-invar}) is clear. For
 (\ref{twist-prod}), we use the normalized Haar measure $dz$ on $\mathbb T$ and observe that for $g,h\in G$,
\begin{align}
\int_K \phi(gkh)\cdot \chi_l(k)\> dk&= \int_{\mathbb T}\int_{\tilde K}
 \phi\left(gd_zkh\right)\chi_l\left(d_z k\right)\> dk\> dz
 =\int_{\mathbb T} \phi\left(gd_z\right)\phi(h)\chi_l\left(d_z\right)\> dz
\notag\\
&= \phi(g)\phi(h)\int_{\mathbb T} \phi\left(d_z \right)\chi_l\left(d_z\right)\> dz \>=\> \phi(g)\phi(h)
\notag \end{align}
as claimed.

For $(2)\Longrightarrow(1)$ consider $\phi$ as in (2). Then by (\ref{twist-invar}),
 $\phi|_K=\chi_l^{-1}$, and  $\phi$ is 
$\tilde K$-biinvariant. Moreover, for $g,h\in G$ and $z\in \mathbb T$,
$\tilde K gd_z\tilde K =\tilde Kd_z g\tilde K$,
and thus
\begin{align}
\phi(g)\phi(h)&=
\int_K \phi(gkh)\cdot \chi_l(k)\> dk= \int_{\mathbb T}\int_{\tilde K}
 \phi\left(gd_zkh\right)\chi_l\left(d_zk\right)\> dk\> dz
\notag\\
 &= \int_{\mathbb T}\int_{\tilde K}
 \phi\left(d_zgkh\right)\chi_l\left(d_z\right)\> dk\> dz=\int_{\tilde K}\phi\left(gkh\right)\> dk\>
\notag \end{align}
as claimed.

The equivalence of (2) and (3) is already mentioned in Section 5.2 of Heckman \cite{HS} 
and can be checked in the same way as for classical spherical functions as it is carried out
 e.g. in Section IV.2 of Helgason \cite{Hel}.
\end{proof}

The   elementary spherical functions $\phi$ on $G$ of type $\chi_l$ for  $l\in\mathbb Z$
are classified by Heckman in Section 5 of \cite{HS}. For a description, we consider the
root system $R:=2\cdot BC_q$ with the positive roots
$$R_+:=\{2e_i,4e_i:\>  i=1,\ldots,q\}\cup\{ 2(e_i-e_j):\>  1\le i<j\le q\}$$
as well as the associated Heckman-Opdam hypergeometric functions according to
\cite{H}, \cite{HS}, \cite{O1}, \cite{O2} which we denote by
$F_{BC_q}(\lambda,k;t)$ with $\lambda\in\mathbb C^q$, $t\in C_q$, and with multiplicity
$k=(k_1,k_2,k_3)$ where the $k_i$ belong to the roots as in the ordering of
$R_+$ above. Note that the  Heckman-Opdam hypergeometric functions $F_{BC_q}(\lambda,k;t)$
exists for all  $\lambda\in\mathbb C^q$, $t\in C_q$ whenever the multiplicity $k$ is 
contained in some open regular subset $K^{reg}\subset\mathbb C^3$. It is well-known 
(see Remark 4.4.3 of \cite{HS}) that, in our notation,
\begin{equation}\label{Kreg}
\{k=(k_1,k_2,k_3):  \> {\it Re}\> k_3\ge0,\>   {\it Re}\> k_1+k_2\ge0\}\subset  K^{reg}.
\end{equation}
 Taking Lemma \ref{euiv-def-spher} into account, we obtain  the following known classification
from
Theorem 5.2.2 of
\cite{HS}:

\begin{theorem}\label{classification-spher-allg}
If $\tilde K$-spherical functions on $G$ are regarded as functions on
$C_q\times\mathbb T$ as above, then the  $\tilde K$-spherical functions on $G$ are
given precisely by
$$\phi_{\lambda,l}^p(t,z)=z^l\cdot\prod_{j=1}^q \cosh^l t_j\cdot
F_{BC_q}(i\lambda,k(p,q,l);t)$$
with $\lambda\in\mathbb C^q$, $l\in\mathbb Z$, and  the multiplicity
$$k(p,q,l)=(p-q-l,\> \frac12 +l,\> 1)\in K^{reg}.$$
\end{theorem}

\begin{example}\label{q1--example-1}
 For $q=1$, the parameter $k_3$ is irrelevant and usually suppressed. 
If one compares
  the one-dimensional example of Heckman-Opdam functions 
on p. 89f of \cite{O1} with the classical definition of the
  Jacobi functions $\phi_\lambda^{(\alpha,\beta)}$ with $\alpha=k_1+k_2-1/2$,
  $\beta=k_2-1/2$ e.g. in \cite{K}, one obtains
$$F_{BC_1}(i\lambda,k;t)=\phi_\lambda^{(\alpha,\beta)}(t);$$
see also Example 3.4 in \cite{R2}. Therefore, by Theorem \ref{classification-spher-allg}, the
$U(p)$-spherical functions on $U(p,1)$ are given by
$$\phi_{\lambda,l}^p(t,z)=z^l\cdot \cosh^l t\cdot \phi_{\lambda}^{(p-1,l)}(t)
\quad\quad(t\ge0, z\in\mathbb T)$$
with $l\in\mathbb Z$ and $\lambda\in\mathbb C$.
This is a classical result of Flensted-Jensen \cite{F}.
\end{example}

Let us summarize the results above. For integers $q\ge 2q$, we obtain from
Proposition \ref{prop-group-torus-convo} that the double coset convolution
of point measures on the double coset space
$$U(p,q)//(U(p)\times SU(q))\simeq C_q\times\mathbb T$$ is given by
\begin{align}\label{twodim-torus-convolution}
&(\delta_{(s,z_1)}*_p \delta_{(t,z_2)})(f):=
\\ &=\frac{1}{\kappa_p}\int_{B_q}\int_{SU(q)} 
f\Bigr(  d(t,s;v,w)  , 
z_1z_2\cdot h(t,s;v,w)
\Bigl)\cdot
\Delta(I_q-w^*w)^{p-2q}
\> dv\> dw\notag
\end{align}
for $s,t\in C_q$, $z_1,z_2\in\mathbb T$ and all bounded continuous functions
$f\in C_b(C_q\times\mathbb T)$ with the abbreviations
\begin{equation}\label{def-d-term}
d(t,s;v,w):=(d_j(t,s;v,w))_{j=1,\ldots,q}:=
 \arcosh(\sigma_{sing}(\sinh\underline t \,w\,\sinh\underline s \,+\,
\>\cosh \underline t\> v\> \cosh \underline s))\in C_q
\end{equation}
and
\begin{equation}\label{def-h-term}
h(t,s;v,w):=
 \Delta(\sinh\underline t \,w\,\sinh\underline s \,+\,
\>\cosh \underline t\> v\> \cosh \underline s).
\end{equation}
It is well known (see e.g.~\cite{J}) that this  double coset convolution can be uniquely
extended in a bilinear, weakly continuous convolution on the Banach space
 $M_b(C_q\times\mathbb T)$ of all bounded regular Borel measures on
$C_q\times\mathbb T$, and that $(C_q\times\mathbb T, *_p)$ forms a commutative hypergroup. 
Moreover, by Theorem \ref{classification-spher-allg}, the functions $\phi_{\lambda,l}^p$
($\lambda\in\mathbb C^q, l\in \mathbb Z$) form the set of all multiplicative functions
on these hypergroups.

\begin{remark}\label{haar-torus}
For integers $p\ge 2q$, the image $\omega_p\in M^+(C_q\times\mathbb T)$
of the Haar measure on $U(p,q)$ under the canonical projection
$$U(p,q)\to U(p,q)//(U(p)\times SU(q))\simeq C_q\times\mathbb T$$
is given, as a measure on   $C_q\times\mathbb T$, by
\begin{equation}\label{haar-measure-torus}
d\omega_p(t,z)=const\cdot
\prod_{j=1}^q \sinh^{2p-2q+1}t_j  \cosh t_j\cdot \prod_{1\le i< j\le q}
\bigl| \cosh(2t_i)-\cosh(2t_j)\bigr|^2\> dt\> dz
\end{equation}
with the Lebesgue measure $dt$ on $C_q$ and the uniform distribution $dz$ on
$\mathbb T$. This measure is, by its construction (see \cite{J}),
 the Haar measure of the hypergroup 
 $(C_q\times\mathbb T, *_p)$.

Eq.~(\ref{haar-measure-torus}) can be derived analogous to the well-known case
 $U(p,q)//(U(p)\times U(q))\simeq C_q$
and is likely to be known; see Section 5 of Heckman \cite{HS}. We shall derive
a more general formula for Haar measures on associated hypergroups on
 $C_q\times\mathbb R$ in Section 4 by using the known Haar measures on the
associated hypergroups on
 $C_q$ due to R\"osler \cite{R2}; this formula
 contains (\ref{haar-measure-torus}). We thus skip a proof here.
\end{remark}

\section{Product formulas on $C_q\times \mathbb R$}

In this section, we extend the product formula (\ref{group-torus-convo}) on 
$C_q\times\mathbb T$
for the spherical functions $\phi_{\lambda,l}^p(t,z)$ of Theorem
\ref{classification-spher-allg} in several ways.

For a fixed dimension parameter $p\ge 2q$,  we first write
 it as a  product formula for functions on the universal covering
 $C_q\times\mathbb R$ of $C_q\times\mathbb T$. For this we define  the functions
\begin{equation}\label{def-psi} 
\psi_{\lambda,l}(t,\theta):=\psi_{\lambda,l}^p(t,\theta):= e^{i\theta l}
\cdot\prod_{j=1}^q \cosh^l t_j\cdot
F_{BC_q}(i\lambda,k(p,q,l);t) \quad\quad(t\in C_q, \> \theta\in\mathbb R)
\end{equation}
with $\lambda\in\mathbb C^q$, $l\in\mathbb Z$, and  the multiplicity
$k(p,q,l)=(p-q-l,\> \frac12 +l,\> 1)$ as above. 
These functions are related to the spherical functions
 $\phi_{\lambda,l}^p$ of the preceding section by
\begin{equation}\label{vgl-psi-phi}
\phi_{\lambda,l}^p(t,e^{i\theta})= \psi_{\lambda,l}(t,\theta) 
 \quad\quad\quad(t\in C_q, \> \theta\in\mathbb R, \> l\in\mathbb Z).
\end{equation}
In a second step we notice that both sides of this product
 formula depend analytically on the parameters
$l$ and $p$ and
 extend the formula to a positive product formula for  all  $l\in\mathbb C$ and all 
 $p\in\mathbb R$ with $p>2q-1$ by some principle of analytic continuation.
For this step we shall employ
 Carlson's theorem on analytic continuation which we 
recapitulate  from \cite{Ti}, p.186,
for the convenience of the reader:

\begin{theorem}\label{continuation} Let $f(z)$ be holomorphic in a neighbourhood of
$\{z\in \mathbb C:{\rm Re\>} z \geq 0\}$ satisfying $f(z) = O\bigl(e^{c|z|}\bigr)$
on $\,{\rm Re\>}  z \geq 0$ for some $c<\pi$. 
If $f(z)=0$ for all nonnegative integers $z$, then $f$ is identically zero for
${\rm Re\>}  z>0$.
\end{theorem}

To state our product formula on  $C_q\times\mathbb R$, we 
  need the fact from complex analysis  that an analytic function
$f:G\to\mathbb C$ on a connected, simply connected set
 $G\subset \mathbb C^n $ with $0\not\in f(G)$ 
admits an analytic logarithm $g:G\to\mathbb C$
with $f=e^g$. In fact, this known result can be shown like in the well-known
 one-dimensional case by using the fact that a closed 1-form is exact on a  
 simply connected domain.
 We also recapitulate that $SU(q)$ is simply connected.
These results and the results of Section 2 lead to the following
 extended product formula:

\begin{theorem}\label{general-twodim-prod-form}
 Let $q\ge1$ be an integer. For all $l\in\mathbb C$ and   $p\in]2q-1,\infty[$,
the functions $\psi_{\lambda,l}^p$ of Eq.~(\ref{def-psi}) 
satisfy the product formula
\begin{align}\label{gen-twodim-prod-form}
&\psi_{\lambda,l}^p(t,\theta_1)\psi_{\lambda,l}^p(s,\theta_2)=\quad\quad \\
&\quad=\frac{1}{\kappa_p}\int_{B_q}\int_{SU(q)} 
\psi_{\lambda,l}^p\Bigr( d(s,t;v,w), 
\theta_1+\theta_2+\rm{Im}\> \ln  h(s,t;v,w)\Bigl)\cdot
\Delta(I_q-w^*w)^{p-2q}
\> dv\> dw.\notag
\end{align}
for all  $s,t\in C_q$, $\theta_1,\theta_2\in\mathbb R$  where
$\kappa_{p} $, $dw$, the functions $d,h$, and other data are defined as in Section 2,
and where $\ln$ denotes the  unique analytic branch of the logarithm of the function
$$(s,t,w,v)\mapsto h(s,t;v,w)=
\Delta(\sinh\underline t \,w\,\sinh\underline s \,+\,
\>\cosh \underline t\> v\> \cosh \underline s)$$
on the simply connected domain $C_q\times C_q\times B_q\times SU(q)$ with 
$\ln\Delta(I_q)=0$.
\end{theorem}

\begin{proof} In a first step take a parameter $l\in\mathbb Z$ and an integer $p\ge 2q$.
In this case, (\ref{gen-twodim-prod-form}) follows immediately from (\ref{vgl-psi-phi}) 
 and the product formula (\ref{group-torus-convo}).

We next observe that both sides of  (\ref{gen-twodim-prod-form}) 
 are analytic in the variables
$p,l,\lambda$.
We now want to employ Carleson's theorem to extend (\ref{gen-twodim-prod-form}) 
to $p\in]2q-1,\infty[$ and $l\in\mathbb C$. However, for this we need
 some exponential growth
estimates for the hypergeometric functions $F_{BC_q}(i\lambda, k(p,q,l); t)$ with respect to
 the parameters $p$ and $l$ in some suitable right half planes, and
 such suitable exponential estimates
are available only for real, nonnegative multiplicities; see Proposition 6.1 of
 \cite{O1}, \cite{Sch}, and Section 3 of \cite{RKV}. We thus  proceed in
 several steps,
 follow the proof of Theorem 4.1 of \cite{R2}, and restrict our attention 
first to a discrete set of spectral parameters $\lambda$ for which  $F_{BC_q}$ is a product
 of the $c$-function and Jacobi polynomials
such that in this case the growth condition can be checked. Carleson's theorem then leads
 to (\ref{gen-twodim-prod-form})  for this discrete set
 of spectral parameters $\lambda$  and all
$p\in]2q-1,\infty[$ and $l\in\mathbb C$.
 In a further step we fix $p\in]2q-1,\infty[$ and $l\in [-1/2, p-q]$ and extend 
 (\ref{gen-twodim-prod-form}) by Carleson's theorem
 to all  spectral parameters $\lambda\in \mathbb C^q$. Finally, 
usual analytic continuation leads to the general result in the theorem for $l\in\mathbb C$.

Let us go into details. We need some notations and facts from \cite{O1}, \cite{O2}, and \cite{HS}.
For our root system $R:=2\cdot BC_q$ with the set $R_+$ of positive roots as in Section 2,
 we define the half sum
of roots
\begin{equation}\label{def-rho}
\rho(k):=\frac{1}{2}\sum_{\alpha\in R_+} k(\alpha)\alpha= 
(k_1 + 2k_2)\sum_{j=1}^q e_j + 2k_3\sum_{j=1}^q(q-j)e_j
\end{equation}
as well as the $c$-function
\begin{equation}\label{c_function} c(\lambda,k) := \prod_{\alpha\in R_+}
\frac{\Gamma(\langle\lambda,\alpha^\vee\rangle +
\frac{1}{2}k(\frac{\alpha}{2}))}{\Gamma (\langle\lambda,\alpha^\vee\rangle +
\frac{1}{2}k(\frac{\alpha}{2}) + k(\alpha))}\cdot \prod_{\alpha\in R_+} \frac{\Gamma
(\langle\rho(k),\alpha^\vee\rangle + \frac{1}{2}k(\frac{\alpha}{2}) +
k(\alpha))}{\Gamma(\langle\rho(k),\alpha^\vee\rangle +
\frac{1}{2}k(\frac{\alpha}{2}))}\end{equation}
with the usual inner product on $\mathbb C^q$ and 
the conventions $\alpha^\vee:=2\alpha/\langle\alpha,\alpha\rangle$
and
 $k(\frac{\alpha}{2}) = 0$ for $\frac{\alpha}{2}\notin R$. Notice that $c$ is meromorphic on 
$\mathbb C^q\times\mathbb C^3$. We now consider the dual root system
  $R^\vee = \{\alpha^\vee: \alpha\in R\}$,
the  coroot lattice  $Q^\vee = \mathbb Z.R^\vee$, and the weight lattice
$P=\{ \lambda\in \mathbb R^q: \langle\lambda,\alpha^\vee\rangle \in \mathbb Z
\,\,\forall\,\alpha\in R\}\,$  of  $R$. Further, denote by
$\, P_+ =\{ \lambda\in P: \langle\lambda,\alpha^\vee\rangle \geq 0
\,\,\forall\,\alpha\in R_+\}\,$  the set of dominant weights associated with $R_+$.
Then, by Eq.~(4.4.10) of \cite{HS} and by \cite{O1}, we obtain for all
 $k\in K^{reg}$ and all $\lambda\in P_+$,
\begin{equation}\label{fbc-jacobi}
 F_{BC_q}(\lambda +\rho(k), k;t) = \,c(\lambda+\rho(k),k)
  P_\lambda(k;t)
\end{equation}
where $c(\lambda,k)$ is the $c$-function \eqref{c_function} which is meromorphic on
$\mathbb C^q\times K$, and where the 
 $P_\lambda$ are the Heckman-Opdam Jacobi polynomials of type $BC_q$.
We now consider the parameters
$$k_{p,l}:=(p-q-l, 1/2+l, 1)\in K^{reg} $$
(see (\ref{Kreg})) as well as the associated half sum of roots
\begin{equation}\label{def-rho-spezial}
\rho(k_{p,l})=
(p-q +l+1)\sum_{j=1}^q e_j\,+ 2 \sum_{j=1}^q (q-j)e_j.
\end{equation}
Using the asymptotics of the gamma function, we now check the growth of
$\, c(\lambda+\rho(k_{p,l}),k_{p,l})\,$
 for fixed $\lambda\in P_+$ and parameters $p,l\to\infty$
 in suitable half planes.
Indeed, by Stirling's formula,
$$\Gamma(z+a)/\Gamma(z)\sim z^a 
\quad\quad\text{for}\quad\quad z\to\infty,\quad {\rm Re}\> z\ge0 . $$
Moreover, for $\rho=\rho(k_{p,l})$,
\begin{align*}
 &c(\lambda + \rho,k) \,=\\
&=\prod_{i=1}^q \frac{\Gamma(\lambda_i + \rho_i)\,\Gamma(\rho_i + k_1)}
{\Gamma(\lambda_i + \rho_i +k_1)\,\Gamma(\rho_i)} \,\cdot 
\prod_{i=1}^q \frac{\Gamma\bigl(\frac{\lambda_i+\rho_i}{2} + \frac{1}{2}k_1\bigr)\,
\Gamma(\frac{\rho_i}{2} + \frac{1}{2} k_1 +
k_2)}{\Gamma\bigl(\frac{\lambda_i+\rho_i}{2} +
\frac{1}{2}k_1+k_2\bigr)\,\Gamma(\frac{\rho_i}{2} + \frac{1}{2} k_1)}\\
&\cdot\prod_{i<j} \frac{\Gamma\bigl(\frac{\lambda_i + \rho_i - \lambda_j
-\rho_j}{2}\bigr)\,
\Gamma\bigl(\frac{\rho_i - \rho_j}{2} + 1\bigr)}
{\Gamma\bigl(\frac{\lambda_i + \rho_i - \lambda_j - \rho_j}{2} + 1\bigr)\,
\Gamma\bigl(\frac{\rho_i-\rho_j}{2}\bigr)}\,\cdot
\,\prod_{i<j}\frac{\Gamma\bigl(\frac{\lambda_i + \rho_i + \lambda_j +\rho_j}{2}\bigr)\,
\Gamma\bigl(\frac{\rho_i + \rho_j}{2} + 1\bigr)}
{\Gamma\bigl(\frac{\lambda_i + \rho_i +\lambda_j + \rho_j}{2} + 1\bigr) \,
\Gamma\bigl(\frac{\rho_i+\rho_j}{2}\bigr)}.
\end{align*}
For $p,l\to\infty$ we obtain that  the first product is
asymptotically equal to $\prod_{i=1}^q 
\bigl(\frac{p+l}{2p}\bigr)^{\lambda_i}$, while the second 
is
asymptotically equal to $\prod_{i=1}^q 
\bigl(\frac{p}{p+l}\bigr)^{\lambda_i/2}$.  
Furthermore,
the third product is independent of $p,l$, and the last product is asymptotically
equal to $1$. In summary, for fixed $\lambda$, the function $c(\lambda + \rho,k)$ behaves like
 $\bigl(\frac{p+l}{p}\bigr)^{a}$ for some $a$, i.e.,
  $c(\lambda + \rho,k)^{-1}$ has polynomial growth.

We now  observe  for $s,t\in C_q$, $w\in B_q$, $v\in SU(q)$, and $
d(t,s;v,w)\in C_q$
that
\begin{equation}\label{submult}
\| d(t,s;v,w)\|_\infty \le \|s\|_\infty+ \|t\|_\infty=s_1+t_1.
\end{equation}
In fact, this follows easily from the submultiplicativity of the spectral norm;
c.f. the proof of Theorem 5.2(1) of \cite{R2}. 
We now use Eq.~(\ref{gen-twodim-prod-form}) for integers $p\ge 2q$ and
$l\in\mathbb Z$, and observe that  for integers $p\ge 2q$ and
$l\in\mathbb Z$ and all $s,t\in C_q$, $\lambda\in P_+$,
 \begin{align}\label{prodpoly}
&\frac{\Bigl(\prod_{j=1}^q \cosh t_j \cdot \cosh s_j\Bigr)^l}{e^{q(t_1+s_1)l}}
P_\lambda(k_{p,l};t)P_\lambda(k_{p,l};s)=\notag \\
\, \,\,& =\frac{1}{\kappa_{p}\cdot c(\lambda+\rho(k_{p,l}),k_{p,l})}
 \int_{B_q}\int_{SU(q)} 
\frac{\Bigl(\prod_{j=1}^q \cosh d(t,s;v,w)\Bigr)^l}{e^{q(t_1+s_1)l}}\cdot
 P_\lambda(k_{p,l}; d(t,s;v,w))\cdot \notag \\
&\quad\quad\quad\quad\quad\quad
 \cdot arg(h(t,s;v,w))
\cdot \Delta(I-w^*w)^{p-2q} dv dw.\end{align}

The Jacobi polynomials $P_\lambda(k;.)$ have rational coefficients in $k$ with respect
 to the monomial  basis $e^\nu,\,\nu\in P$; see Section 11 of \cite{M} or
 the explicit determinantal construction in Theorem 5.4 of \cite{DLM}.
Carlson's theorem now yields that formula \eqref{prodpoly} holds for all $l\in
\mathbb C$.
Moreover, as derived in the proof of Theorem 3.6 of \cite{R1}, the normalized integral
\[ \frac{1}{|\kappa_{p}|} \int_{B_q} \vert\Delta(I-w^*w)^{p-2q}\vert dw\]
converges exactly if $\,\text{Re}\, p >2q -1$ and 
is  of polynomial growth for $p\to\infty$ in the right halfplane
$\{p\in\mathbb C:\> {\rm Re}\> p\ge 2q\}$.
Thus for fixed $t,s,$ both sides of \eqref{prodpoly} are holomorphic  and of 
polynomial growth for $p\to\infty$  in this halfplane.
 Moreover, they coincide for all  integers $ p\geq 2q$. Another
 application of Carlson's theorem yields that \eqref{prodpoly} also 
holds for all $p$ in this halfplane. 
This proves the stated result for all spectral parameters
 $\lambda + \rho(k)$ with $\lambda\in P_+$. 

\smallskip 
In the final step we
 extend the product formula with respect to the spectral parameter. For this
 we fix
 $s,t\in C_q$, $p\ge 2q$ a real number, and $l\in[-1/2,q-p]$. Then
 $k=k_{p,l}$ is nonnegative, and we have
 the estimate 
\[ |F_{BC_q}(\lambda,k;t)|\,\leq \,|W|^{1/2} e^{{max}_{w\in W}\text{Re}\langle
  w\lambda,t\rangle} \]
by  Proposition 6.1 of  \cite{O1}. We now can proceed precisely as in the last
step of the proof of Theorem 4.1 on pp. 2791f of \cite{R2}.  As in
Eq.~(\ref{prodpoly}), we can rewrite (\ref{gen-twodim-prod-form}) with 
some suitable exponential growth correction which ensures that  by a $q$-fold application of
Carlson's theorem, (\ref{gen-twodim-prod-form}) 
can be extended to all $\lambda\in \mathbb C^q$. We skip the details.

 In a final step, usual
analytic continuation yields the theorem for all  $l\in \mathbb C^q $, which
completes the proof.
\end{proof}

It should be noticed that the precise choice of the complex logarithm in
(\ref{gen-twodim-prod-form})  is not essential for this product formula, as
it has no influence on the analyticity of the formulas  above
 with respect to  $p,l,\lambda$. However, our choice of 
the complex logarithm in (\ref{gen-twodim-prod-form}) will be essential in
the next section when we prove that 
(\ref{gen-twodim-prod-form}) induces an associative convolution algebra on the
Banach space of all bounded signed probability measures on $C_q\times\mathbb
R$.

\begin{remark} For $p=2q-1$, a degenerate version of the product formula 
(\ref{gen-twodim-prod-form}) is available. For this we need some notations and
  facts. We here follow Section 3 of \cite{R1} and Remark 2.14 of \cite{V2}
  where also such limit cases of product formulas for Bessel and Laguerre functions on
  matrix cones were considered.

We fix the dimension $q$ and consider the matrix ball 
$B_q:=\{w\in M^{q,q}(\mathbb C):\> w^*w\le I_q\}$ as above as well as the ball 
$B:=\{y\in\mathbb C^q:\> \|y\|_2<1\}$ and the sphere 
$S:=\{y\in\mathbb C^q:\> \|y\|_2=1\}$.
By Lemma 3.6 and Corollary 3.7 of \cite{R1}, the mapping $P:B^q\to B_q$ from
the direct product $B^q$ to  $B_q$
with
\begin{equation}\label{trafo-P}
 P(y_1, \ldots, y_q):= \begin{pmatrix}y_1\\y_2(I_q-y_1^*y_1)^{1/2}\\ 
\vdots\\
  y_q(I_q-y_{q-1}^*y_{q-1})^{1/2}\cdots (I_q-y_{1}^*y_{1})^{1/2}\end{pmatrix}
\end{equation}
establishes a diffeomorphism such that the image of the measure 
$$\Delta(I_q-w^*w)^{p-2q}dw$$
under  $P^{-1}$ is given by $\,\prod_{j=1}^{q}(1-\|y_j\|_2^2)^{p-q-j}dy_1\ldots
dy_q$. Therefore, for $p>2q-1$, the product formula 
(\ref{gen-twodim-prod-form}) may be rewritten as
\begin{align}\label{convo-mod}
\psi_{\lambda,l}^p(t,\theta_1)\psi_{\lambda,l}^p(s,\theta_2)
=\frac{1}{\kappa_p}\int_{B^q}\int_{SU(q)} &
\psi_{\lambda,l}^p\Bigr(  d(t,s;v,P(y))   ,
\theta_1+\theta_2+ {\rm{Im}}\> \ln h(t,s;v,P(y))\Bigl)\cdot \notag  \\
&
\cdot\prod_{j=1}^{q}(1-\|y_j\|_2^2)^{p-q-j}dy_1\ldots
dy_q\> dw
\end{align}
with $y=(y_1,\ldots,y_q)\in B^q$, 
where $dy_1,\ldots,dy_q$ means integration with respect to the Lebesgue
measure on $\mathbb C^q$.
 Moreover,  for $p\downarrow 2q-1$, we obtain from (\ref{convo-mod}) 
by continuity the following degenerated
 product formula:
\begin{align}\label{convo-degen}
\psi_{\lambda,l}^{2q-1}(t,\theta_1)\psi_{\lambda,l}^{2q-1}(s,\theta_2)=
\frac{1}{\kappa_{2q-1}}\int_{B^q}\int_{SU(q)} &
\psi_{\lambda,l}^{2q-1}\Bigr(  d(t,s;v,P(y))  , 
\theta_1+\theta_2+  {\rm{Im}}\> \ln h(t,s;v,P(y))\Bigl)\cdot \notag  \\
&
\cdot\prod_{j=1}^{q-1}(1-\|y_j\|_2^2)^{q-1-j}dy_1\ldots
dy_{q-1}\> d\sigma(y_q)\> dw
\end{align}
where $\sigma\in M^1(S)$ is the uniform distribution on $S$ and
$$\kappa_{2q-1}:=\int_{B^q}\prod_{j=1}^{q-1}(1-\|y_j\|_2^2)^{q-1-j}dy_1\ldots
dy_{q-1}\> d\sigma(y_q).$$
\end{remark}

\section{Commutative hypergroups on $C_q\times\mathbb R$ }

The positive product formulas (\ref{gen-twodim-prod-form}) for real parameters 
$p>2q-1$ and (\ref{convo-degen}) for $p=2q-1$ lead to a continuous series of
probability-preserving convolution algebras on $C_q\times\mathbb R$
parametrized by $p\ge 2q-1$. In fact, these convolutions form 
commutative
hypergroups, which have the functions $\psi_{\lambda,l}^{p}$
($\lambda\in\mathbb C^q, l\in\mathbb C$) as multiplicative functions.
For $q=1$, our convolution structures are closely related to those studied by
Trimeche \cite{T}.

Before going into details, we briefly recapitulate some notions from hypergroup theory. 
For more details
we refer to \cite{J} and the monograph \cite{BH}.
Hypergroups generalize the convolution of bounded measures on locally compact groups
such that the convolution  $\delta_x*\delta_y$  of two point measures 
$\delta_x, \delta_y$ is a probability measure with compact support, but not
 necessarily  a point measure.

\begin{definition} A hypergroup is a locally compact Hausdorff space $X$ with a
  weakly continuous, associative, bilinear convolution $*$ on the Banach space $M_b(X)$
  of all bounded regular Borel measures on $X$ such that the following
  properties hold:
\begin{enumerate}
\item[\rm{(1)}] For all $x,y\in X$,  $\delta_x*\delta_y$  is a compactly
  supported probability measure on $X$ such that the support 
  $\text{supp}\>(\delta_x*\delta_y)$ depends continuously on $x,y$ with respect to the
  Michael topology on the space of all compacta in $X$ (see \cite{J} for details).  
\item[\rm{(2)}] There exists a neutral element $e\in X$ with $\delta_x*\delta_e=
\delta_e*\delta_x=\delta_x$ for all $x\in X$. 
\item[\rm{(3)}] There exits a continuous involution $x\mapsto\bar x$ on $X$
  such that for all $x,y\in X$, $e\in \text{supp}\> (\delta_x*\delta_y)$ holds
  if and only if $y=\bar x$. 
\item[\rm{(4)}] If for $\mu\in M_b(X)$, $\mu^-$ is the
  image of $\mu$ under the involution, we require that
  $(\delta_x*\delta_y)^-= \delta_{\bar y}*\delta_{\bar x}$ for all $x,y\in X$.
\end{enumerate}
\end{definition}

Due to weak continuity and bilinearity, the convolution of arbitrary bounded
measures on a  hypergroup is determined uniquely by the convolution of point
measures. 

A  hypergroup is called commutative if so is the convolution.
We recapitulate from \cite{J}  that for a Gelfand pair $(G,K)$, the double coset
 convolution on the double coset space $G//K$ forms a commutative hypergroup.

 For a  commutative hypergroup we define the space
$$\chi(X) = \{\varphi\in C(X): \,\varphi\not\equiv 0, \, \varphi( x* y ):=
(\delta_x*\delta_y)(\varphi) = \varphi(x)\varphi(y) \,\, \forall\,  x,y\in X\} $$
of all nontrivial continuous  multiplicative functions  on $X$, as well as the dual space
$$ \widehat X :=\{\varphi\in \chi(X): \,\varphi \text{ is bounded and }\, 
\varphi(\overline x)= \overline{\varphi(x)}\,\, \forall\,  x\in X\}.$$ 
The elements of $\widehat X$ are called characters. 

Using the  positive product formulas (\ref{gen-twodim-prod-form}) and 
(\ref{convo-degen}) for 
$p>2q-1$ and $p=2q-1$ respectively, 
we now introduce the convolutions of point measures on $X:=C_q\times\mathbb R$ depending on $p$:
For $s,t\in C_q$ and $\theta_1,\theta_2\in\mathbb R$, we define the probability measures 
$(\delta_{(s,\theta_1)}*_p \delta_{(t,\theta_2)})$ with compact supports by
\begin{align}\label{gen-twodim-convolution}
&(\delta_{(s,\theta_1)}*_p \delta_{(t,\theta_2)})(f):=
\\ &=\frac{1}{\kappa_p}\int_{B_q}\int_{SU(q)} 
f\Bigr(  d(t,s;v,w)  , 
\theta_1+\theta_2+{\rm{Im}}\>\ln h(t,s;v,w)
\Bigl)\cdot
\Delta(I_q-w^*w)^{p-2q}
\> dv\> dw\notag
\end{align}
for $p>2q-1$, and by 
\begin{align}\label{convo-degen-2}
&(\delta_{(s,\theta_1)}*_{2q-1} \delta_{(t,\theta_2)})(f):=
\\ &=\frac{1}{\kappa_{2q-1}}\int_{B^q}\int_{SU(q)} 
f\Bigr(  d(t,s;v,P(y))  , 
\theta_1+\theta_2+ {\rm{Im}}\>\ln h(t,s;v,P(y))\Bigl)\cdot 
\notag\\ &\quad\quad\quad\cdot
\prod_{j=1}^{q-1}(1-\|y_j\|_2^2)^{q-1-j}dy_1\ldots
dy_{q-1}\> d\sigma(y_q)\> dw\notag
\end{align}
 for $p=2q-1$ for all $f\in C(C_q\times\mathbb R)$,
 where the functions  $d, h$ and the other data are given as in
 Sections 2 and 3.

\begin{theorem}
Let $q\ge 1$ be an integer and $p\in[2q-1,\infty[$. 
 Then $*_p$ can be extended uniquely to a bilinear,
 weakly continuous convolution on the Banach space $M_b( C_q\times \mathbb R)$. 
This convolution is associative, and $(C_q\times \mathbb R, *_{p})$ is a 
commutative hypergroup with $(0,0)$ as identity and with the involution
 $\overline{(r,a)}:= (r,-a)$. 
\end{theorem}

\begin{proof} It is clear by the definition of the convolution that 
 the mapping
$$(C_q\times\mathbb R)\times(C_q\times\mathbb R)\to M_b(C_q\times\mathbb R), \quad
((s,\theta_1), (t,\theta_2))\mapsto \delta_{(s,\theta_1)} *_{p} \delta_{(t,\theta_2)} $$
is probability preserving and weakly
continuous. It is now standard (see \cite{J}) to extend  this convolution
uniquely in a bilinear, weakly
continuous way to a probability preserving convolution on  $M_b(C_q\times \mathbb R)$.

To prove commutativity, it suffices to  consider point measures.
But in this case, for $p>2q-1$,
 commutativity follows easily by using transposition of matrices and the fact that 
the integration in (\ref{gen-twodim-prod-form})
 over $w\in B_q$ and $v\in SU(q)$ remains invariant under transposition. The commutativity  
 for $p>2q-1$ yields in the limit also commutativity for $p=2q-1$.

We now turn to associativity. Again it is sufficient to consider point measures.
We first consider integers $p\ge 2q$. In this case, we first identify
$C_q\times\mathbb T$ with the double coset space
 $U(p,q)//( U(p)\times SU(q))$ as in Section 2 where the associated double coset
 convolution on $C_q\times\mathbb T$ is given by the product formula
(\ref{group-torus-convo}). After bilinear, weakly
continuous extension, this double coset convolution is associative by its very construction.
We now compare  convolution products w.r.t. this convolution on  $C_q\times\mathbb T$ with 
the corresponding  one defined by
 (\ref{gen-twodim-prod-form}) on  $C_q\times\mathbb R$ for
 $(r,\theta_1),(s,\theta_2),(t,\theta_3)\in C_q\times\mathbb R$ for small $r,s,t\in C_q$.
Taking (\ref{submult}) into account, we see readily that the complex logarithm in
(\ref{gen-twodim-prod-form}) for the triple products 
$$(\delta_{(r,\theta_1)}*_{p} \delta_{(s,\theta_2)})*_{p} \delta_{(t,\theta_3)}
\quad\text{and}\quad
\delta_{(r,\theta_1)}*_{p} (\delta_{(s,\theta_2)}*_{p} \delta_{(t,\theta_3)})$$
is just the usual main branch of the logarithm on the open right halfplane
$\{z:\> \text{Re}\>z>0\}$
for all integration variables
 in $B_q$ and $SU(q)$ respectively. Using this elementary logarithm, we see immediately that
the associativity of the convolution on  $C_q\times\mathbb T$ implies the associativity
on $ C_q\times\mathbb R$ for small $r,s,t\in C_q$. 
We now extend the associativity to arbitrary 
$(r,\theta_1),(s,\theta_2),(t,\theta_3)\in C_q\times\mathbb R$. For this we use
(\ref{submult}) and find an open, relatively compact set $K\subset C_q\times\mathbb R$ with
$$\text{supp}\>((\delta_{(\tilde r,\tilde \theta_1)}*_{p} \delta_{(\tilde s,\tilde \theta_2)})*_{p}
\delta_{(\tilde t,\tilde \theta_3)})\cup
\text{supp}\>(\delta_{(\tilde r,\tilde \theta_1)}*_{p} (\delta_{(\tilde s,\tilde \theta_2)}*_{p}
\delta_{(\tilde t,\tilde \theta_3)}))\subset K$$
for all $\tilde r,\tilde s,\tilde t\in C_q$ and 
$\tilde \theta_1,\tilde \theta_2,\tilde \theta_3\in\mathbb R$ with
$$\|\tilde r\|_\infty,\|\tilde s\|_\infty,\|\tilde t\|_\infty\le
 2\text{max}(\| r\|_\infty,\|s\|_\infty,\|t\|_\infty)
\quad\text{and}\quad
|\tilde \theta_1|,|\tilde \theta_2|,|\tilde \theta_3|\le
 2\text{max}(| \theta_1|,| \theta_2|,| \theta_3|).$$
 Let $f\in C_c( C_q\times\mathbb R)$ be a continuous function with compact
 support which is analytic on $K$. Then by analyticity of the product formulas
 (\ref{gen-twodim-prod-form}) and (\ref{convo-degen}) w.r.t.~the variables in 
$C_q\times\mathbb R$,
\begin{equation}\label{proof-asso}
((\delta_{(\tilde r,\tilde \theta_1)}*_{p} \delta_{(\tilde s,\tilde
    \theta_2)})*_{p}
 \delta_{(\tilde t,\tilde \theta_3)})(f)
\quad\text{and}\quad
(\delta_{(\tilde r,\tilde \theta_1)}*_{p} (\delta_{(\tilde s,\tilde
  \theta_2)}*_{p} \delta_{(\tilde t,\tilde \theta_3)}))(f)
\end{equation}
are analytic in the variables 
 $\tilde r,\tilde s,\tilde t,\tilde \theta_1,\tilde \theta_2,\tilde \theta_3$
where both expressions are equal for small  $\tilde r,\tilde s,\tilde t$.
Therefore, they are equal in general for all such functions $f$.
 As both measures have compact support, a
Stone-Weierstrass argument  leads to the general associativity for
integers $p\ge 2q-1$. We now extend the associativity to arbitrary $p>2q-1$ by
Carleson's theorem. For this we  compare both sides of (\ref{proof-asso}) for
functions $f$ as above which is analytic also in the variable $p\in\mathbb C$
with $\text{Re}\> p>2q-1$ where the boundedness condition in Carleson's
theorem can be obtained from the fact that 
\[ \frac{1}{|\kappa_{p}|} \int_{B_q} \vert\Delta(I-w^*w)^{p-2q}\vert dw\]
is  of polynomial growth for $p\to\infty$ in the right halfplane
$\{p\in\mathbb C:\> {\rm Re}\> p\ge 2q\}$; see Theorem 3.6 of \cite{R1}. This
completes the proof of associativity.

For the remaining hypergroup axioms we first notice that $(0,0)$ is obviously the neutral
element. Moreover, for $(t,\theta)\in C_q\times\mathbb R$ we have 
$$(0,0)\in\text{supp}\> (\delta_{t,\theta}*_p \delta_{t,-\theta})$$ by taking
the integration variables $v\in SU(q)$ as the identity matrix $I_q$ 
and $w:=-I_q\in B_q$ 
in the convolution (\ref{gen-twodim-convolution}) and  $y_1:=-e_1,\ldots,y_q:=-e_q$
 for the  usual unit vectors in $  \mathbb C^q$  in  (\ref{convo-degen-2}).

We next check the converse part of axiom (3) of a hypergroup.
For this take $(s,\theta_1),(t,\theta_2)\in C_q\times\mathbb R$ with 
$(0,0)\in\text{supp}\> (\delta_{s,\theta_1}*_p\delta_{t,\theta_2})$.
As the support is independent of $p\in ]2q-1,\infty[$ with
$$\text{supp}\>(\delta_{s,\theta_1}*_{2q-1}\delta_{t,\theta_2})\subset 
\text{supp}\>(\delta_{s,\theta_1}*_{2q}\delta_{t,\theta_2})$$
by (\ref{gen-twodim-convolution}) and (\ref{convo-degen-2}), we may restrict
our attention to integers $p\ge 2q$. 
In this case we now compare the convolution (\ref{gen-twodim-convolution})
 on $C_q\times \mathbb R$ with the  convolution (\ref{twodim-torus-convolution})
on  $C_q\times\mathbb T$ which is the double coset convolution for 
$U(p,q)//(U(p)\times SU(q))$. As here axiom (3) is available automatically,
we conclude from our assumption that $s=t$ and 
$\theta_1-\theta_2\in2\pi\mathbb Z $ holds.
For the proof of $\theta_1=-\theta_2$, we  analyze  (\ref{twodim-torus-convolution})
more closely:  Recapitulate from Section 2 that the identification 
$C_q\times\mathbb T\simeq U(p,q)//(U(p)\times SU(q))$ is done via the representatives
 $a_{t,z}\in U(p,q)$ ($t\in C_q$, $z\in\mathbb T$) of double cosets. It is clear that for
all $t\in C_q$ and $z\in\mathbb T$, the matrix
$$J:=\begin{pmatrix} -I_{p}&0\\ 0 & I_q  \end{pmatrix}\in U(p)\times SU(q)$$
is the only element of  $U(p)\times SU(q)$ with $a_{t,z}\cdot J\cdot a_{t,z^{-1}}=I_{p+q}=a_{0,1}$.
By the proof of Proposition \ref{prop-group-torus-convo} this means that for $v\in SU(q)$ and 
$w=0\sigma_0^*u\sigma_0\in B_q$ with $\sigma_0$ as in (\ref{def-sigma-0}), we have
$$d(t,t;v,w)=0 \quad\text{and}\quad arg\>  h(t,t;v,w) =
 arg\>\Delta(\sinh\underline t \,w\,\sinh\underline t \,+\,
\>\cosh \underline t\> v\> \cosh \underline t)=1$$
 only for $v=I_q$ and $w=-I_q$. In this case, 
$d(t,t;I_q,-I_q)=0$ and $h(t,t;I_q,-I_q)=1$, where for all $t\in C_q$ obviously
 the branch of the complex logarithm in the product formula  (\ref{gen-twodim-convolution})
satisfies  $\ln h(t,t;I_q,-I_q)=1$. This proves  $\theta_1=-\theta_2$ above and completes 
the proof of axiom (3).

 Furthermore, axiom (4) is clear,
 and the continuity of the supports of convolution
products can be  checked in a straightforward, but technical way. We skip the details.
\end{proof}

 We next turn to subgroups of the commutative hypergroups  $(C_q\times \mathbb R, *_{p})$ for $p\ge 2q-1$.
For this we recapitulate that a closed non-empty subset $H\subset C_q\times \mathbb R$ is a subhypergroup if 
 ${\rm supp}\> (\delta_x*_p\delta_{\bar y})\subset H$ holds for all $x,y\in H$.
Moreover, $H$ is called a subgroup, if the convolution restricted to $H$ is the convolution of a group 
structure on $H$. It is clear from
(\ref{gen-twodim-prod-form}) and (\ref{convo-degen}), that $\{0\}\times \mathbb R$ is a subgroup 
of $(C_q\times \mathbb R, *_{p})$ which is isomorphic to the group $(\mathbb R,+)$. 

Now let $H$ be a subgroup of a commutative hypergroup $(X,*)$. 
Then the cosets
 $$x*H:=\bigcup_{y\in H}{\rm supp}\> (\delta_x*_p\delta_{ y}) \quad\quad (x\in X) $$
form a disjoint decomposition of $X$, and the quotient $X/H:=\{x*H: \> x\in X\}$
 is again a locally compact Hausdorff space with respect to the quotient topology. Moreover,
\begin{equation}\label{quotient-allg}
(\delta_{x*H}*\delta_{y*H})(f):=\int_X f(z*H)\> d(\delta_{x}*\delta_{y})(z)
\quad\quad\quad (x,y\in X, \> f\in C_b(X/H)),
\end{equation}
establishes a well-defined quotient convolution and an associated quotient hypergroup $(X/H,*)$.
For these quotient convolutions we refer to \cite{J}, \cite{R} and \cite{V1}.
 We now apply this concept to the subgroups $\{0\}\times \mathbb R$ and $\{0\}\times \mathbb Z$ of our hypergroups 
 $(C_q\times \mathbb R, *_{p})$ for $p\ge 2q-1$. By
 (\ref{quotient-allg}) and 
 (\ref{gen-twodim-prod-form}) we can identify the quotient spaces in the obvious way with
$C_q$ and $C_q\times\mathbb T$ respectively, and we obtain immediately:

\begin{lemma}\label{bsp-quotienten} Let $p> 2q-1$.
\begin{enumerate}
\item[\rm{(1)}] $(C_q\times \mathbb R)/(\{0\}\times \mathbb R)\simeq C_p$ is a commutative  hypergroup with 
the convolution
\begin{align}\label{gen-onedim-convolution}
(\delta_s*_p \delta_t)(f):=
\frac{1}{\kappa_p}\int_{B_q}\int_{SU(q)} 
f\Bigr(  d(t,s;v,w) 
\Bigl)\cdot
\Delta(I_q-w^*w)^{p-2q}
\> dv\> dw
\end{align}
for $s,t\in C_q$ and $f\in C_b(C_q)$. 
This is precisely the hypergroup studied in Section 5 of \cite{R2}.
 For integers $p\ge 2q$, this is just the double coset hypergroup $U(p,q)/(U(p)\times U(q))$.
\item[\rm{(2)}] $(C_q\times \mathbb R)/(\{0\}\times \mathbb Z)\simeq C_p\times\mathbb Z$ is a commutative 
 hypergroup with 
the convolution $*_p$ as defined in Eq.~(\ref{twodim-torus-convolution}), but here 
for arbitrary real numbers $p>2q-1$.
In particular, for integers $ p\ge 2q$, this is just the double coset hypergroup $U(p,q)/(U(p)\times SU(q))$.
\end{enumerate}
\end{lemma}

For  $p= 2q-1$, a corresponding result holds on the basis of (\ref{convo-degen}).

By using Weil's integral formula for Haar measures on hypergroups (see \cite{Her} and \cite{V1}),
 we now can determine the Haar measures on the hypergroups $(C_q\times \mathbb R, *_p)$ from the known 
Haar measures on the hypergroups of Lemma \ref{bsp-quotienten}(1) in Theorem 5.2 of \cite{R2}.
For this recapitulate that each commutative hypergroup $(X,*)$ 
admits a (up to a multiplicative constant unique)
Haar measure $\omega_X$, which is characterized by the condition $\omega_X(f)=\omega_X(f_x)$ for
all  continuous functions $f\in C_c(X)$ with compact support and $x\in X$, where the translate
 $f_x\in C_c(X)$ is given by
 $f_x(y):=(\delta_y*\delta_x)(f)$.

\begin{proposition}
For $p\ge 2q-1$, the Haar measure of the commutative hypergroup
$(C_q\times\mathbb R, *_p)$, is given by
\begin{equation}\label{haar-measure-full}
d\omega_p(t,\theta)=const\cdot
\prod_{j=1}^q \sinh^{2p-2q+1}t_j  \cosh t_j\cdot \prod_{1\le i< j\le q}
\bigl| \cosh(2t_i)-\cosh(2t_j)\bigr|^2\> dt\> d\theta
\end{equation}
with the Lebesgue measure $dt$ on $C_q$ and the Lebesgue measure $d\theta$ on $\mathbb R$.

Moreover, for $p\ge 2q-1$, the Haar measure of the commutative quotient hypergroup
 $(C_q\times \mathbb R)/(\{0\}\times \mathbb Z)\simeq C_p\times \mathbb T$ of the preceding lemma
is given precisely by the measure of Eq.~(\ref{haar-measure-torus}) on $C_q\times\mathbb T$.
\end{proposition}

\begin{proof}
Let $H$ be a subgroup of a commutative hypergroup $(X,*)$, and let $\omega_H$ and $\omega_{X/H}$ 
be Haar measures of the group $H$ and the quotient hypergroup $X/H$ respectively. 
Then, by \cite{Her} and \cite{V1}, 
for $f\in C_c(X)$ the function $T_Hf(x*H):=\int_H (\delta_x*\delta_h)\> d\omega_H(h)$
 establishes a well-defined function $T_Hf\in C_c(X/H)$, and
 Weil's  formula
$$\omega_X(f):=\int_{X/H} T_Hf(x*H)\> d\omega_{X/H}(x*H)$$ 
defines (up to a multiplicative constant) the Haar measure $\omega_X$ of  $(X,*)$.
If we apply this construction to our  commutative hypergroup
$(C_q\times\mathbb R, *_p)$ with the subgroup $\{0\}\times\mathbb R$, and take the Haar measure in Theorem 5.2(2)
of \cite{R2} for the quotient $(C_q\times \mathbb R)/(\{0\}\times \mathbb R)\simeq C_p$, the first part of the
proposition is clear.

The second statement follows immediately from the first statement and the fact that a Haar measure on the quotient
 $(C_q\times \mathbb R)/(\{0\}\times \mathbb Z)\simeq C_p\times\mathbb T$ 
is given as the image of a Haar measure on $C_q\times \mathbb R$ under the canonical projection.
\end{proof}

\begin{remark}
\begin{enumerate}
\item[\rm{(1)}] By the preceding results, the functions $\psi_{\lambda,l}^p$
  ($\lambda\in\mathbb C^q, l\in \mathbb C$) are continuous multiplicative
  functions on the hypergroups $(C_q\times\mathbb R, *_p)$ for $p\ge 2q-1$. 
 We conjecture that in fact each continuous multiplicative
  function on  $(C_q\times\mathbb R, *_p)$ has this form. In fact, one has to
  prove similar to Lemma 5.3 of \cite{R2} that continuous multiplicative
  functions are eigenfunctions of a corresponding family of differential
  operators discussed in Section I.5 of \cite{HS}.
\item[\rm{(2)}] The multiplicative functions $\psi_{\lambda,l}^p$ satisfy
  several symmetry conditions in the spectral variables 
which are immediate consequences of
  corresponding symmetries for $F_{BC_q}$ in \cite{HS}, \cite{O1},
  \cite{O2}. In particular, similar to  \cite{R2}, we have:

 For $\lambda,\tilde\lambda\in\mathbb C^q$,
  $l,\tilde l\in \mathbb C$, and the
  Weyl group $W_q$ of type $B_q$ acting on $\mathbb C^q$,
$$\psi_{\lambda,l}^p=\psi_{\tilde \lambda,\tilde l}^p
  \quad\text{on}\quad C_q
\quad\quad\Longleftrightarrow\quad\quad \tilde\lambda\in W\lambda, \quad \tilde
l=l.$$
Moreover,
$$\overline{\psi_{\lambda,l}^p}= \psi_{\bar\lambda,-l}^p.$$
In particular,  $\psi_{\lambda,l}^p$ satisfies
\begin{equation}\label{symmetry-cond}
\psi_{\lambda,l}^p((t,\theta)^-)=\overline{\psi_{\lambda,l}^p(t,\theta)}
\quad\quad\text{for
all}\quad
(t,\theta)\in C_q\times\mathbb R
\end{equation}
 if and only if $l\in\mathbb R$ and
$\bar\lambda\in  W\lambda$
 holds.
\item[\rm{(3)}] It is an interesting task to determine the dual space
 $(C_q\times\mathbb R)^\wedge$ which consists of all bounded, continuous
  multiplicative  functions satisfying (\ref{symmetry-cond}). For the
  corresponding hypergroups on $C_q$, we refer to  \cite{R2} and  \cite{NPP}
  for this problem.
\end{enumerate}
\end{remark}

\section{Product formulas for  Heckman-Opdam
functions}

We fix the dimension $q\ge1$, a real parameter $p>2q-1$ and
 some index $l\in\mathbb R$. We know from Section 2 that the functions $\psi_{\lambda,l}$ with
$$\psi_{\lambda,l}^p(t,0)=
\prod_{j=1}^q \cosh^l t_j\cdot
F_{BC_q}(i\lambda,k(p,q,l);t)$$
satisfy the product formula (\ref{gen-twodim-prod-form}). We now apply
$$\prod_{j=1}^q \cosh^l d_j(s,t;v,w)\cdot e^{il\cdot \rm{Im}\> \ln  h(s,t;v,w)}=
 h(s,t;v,w)^l=\Delta(\sinh\underline t \,w\,\sinh\underline s \,+\,
\>\cosh \underline t\> v\> \cosh \underline s)^l$$
for $s,t\in C_q$, $w\in B_q$, $v\in SU(q)$ and the analytical branch of the $l$-th power function 
associated with the analytical branch of the logarithm in the setting of  (\ref{gen-twodim-prod-form}).
This branch will be taken always from now on.
This leads immediately to the following product formula for the hypergeometric functions 
$\phi_\lambda^{p,l}(t):= F_{BC_q}(i\lambda, k(p,q,l);t)$:

\begin{theorem}
Fix an integer $q\ge1$,  $p\in [2q-1,\infty[$, and all $l\in\mathbb R$. Then  the functions  
$\phi_\lambda^{p,l}$ satisfy the product formula
\begin{align}\label{prod-formel-allg-hypergeo}
\phi_\lambda^{p,l}(s)\cdot\phi_\lambda^{p,l}(t)&= 
\frac{1}{\kappa_p\prod_{j=1}^q \Bigl( \cosh t_j\cdot\cosh s_j\bigr) ^{l}}  \cdot\\
&\quad \cdot
\int_{B_q}\int_{SU(q)}  
\phi_\lambda^{p,l}(  d(t,s;v,w))\cdot  {\rm Re}(h(t,s;v,w)^l) \cdot
\Delta(I_q-w^*w)^{p-2q}
\> dv\> dw\notag
\end{align}
for $s,t\in C_q$ and all $\lambda\in\mathbb C$.
\end{theorem}

\begin{proof} Our considerations above lead to
\begin{align}
\phi_\lambda^{p,l}(s)\cdot\phi_\lambda^{p,l}(t)&= 
\frac{1}{\kappa_p\prod_{j=1}^q \Bigl( \cosh t_j\cdot\cosh s_j\bigr) ^{l}} \cdot \notag\\
&\quad \cdot
\int_{B_q}\int_{SU(q)}  
\phi_\lambda^{p,l}(  d(t,s;v,w))\cdot  h(t,s;v,w)^l \cdot
\Delta(I_q-w^*w)^{p-2q}
\> dv\> dw.\notag
\end{align}
for $s,t\in C_q$, $\lambda\in\mathbb C$.
Now take $\lambda\in\mathbb R^q$ in which case $\phi_\lambda^{p,l}$ is real on $C_q$. Therefore, 
taking real parts above, we obtain the product formula of the theorem for  $\lambda\in\mathbb R^q$.
The general case follows by analytic continuation.
\end{proof}

We next present a condition on $l$ which ensures positivity of the product formula
(\ref{prod-formel-allg-hypergeo}) for all $s,t\in C_q$. It is based on the following:

\begin{lemma}\label{pos-l}
For all $l\in\mathbb R$  with $|l|\le 1/q$ and all $s,t\in C_q$, $w\in B_q$, $v\in SU(q)$,
$${\rm Re}\> \Bigl((\Delta(\sinh\underline t \,w\,\sinh\underline s \,+\,
\>\cosh \underline t\> v\> \cosh \underline s))^l\Bigr)\ge0.$$
\end{lemma}

\begin{proof} 
We have
\begin{equation}\label{det-gl}
\Delta(\sinh\underline t \,w\,\sinh\underline s \,+\,
\>\cosh \underline t\> v\> \cosh \underline s)=
\Delta(\cosh \underline t\cdot\cosh \underline s)\cdot
\Delta(\tilde w+I_q)
\end{equation}
for the matrix $\tilde w:= v^{-1}\cdot \tanh \underline s  \,w\, \tanh \underline t$.
We now check $\tilde w\in B_q$, i.e., $\tilde w^*\tilde w\le I_q$.
In fact, this is equivalent to
 $\tanh \underline t \,w^* \,\tanh^2 \underline s \,w\,\tanh \underline t\le I_q$,
which is clearly a consequence of $w^* \,\tanh^2 \underline s \,w\le I_q$ which is 
obviously correct.

As all eigenvalues $\tau\in\mathbb C$ of a matrix $\tilde w\in B_q$ satisfy
 $|\tau|\le1$, we obtain that all eigenvalues of $\tilde w+I_q$ are contained
 in $\{z\in\mathbb C:\> {\rm Re}\>z\ge0\} $. The lemma now follows from (\ref{det-gl}).
\end{proof}

We notice that it can be easily seen that Lemma \ref{pos-l}
 is not correct for a larger range of parameters $l\in\mathbb R$.
 It is however unclear for which precise range of parameters $l\in \mathbb R$
 there is a positive product formula for the functions $\phi_\lambda^{p,l}$. 
We expect that this range depends on $q$ and $p$; see also the example for $q=1$ below.

We also remark that the results above for $p>2q-1$ are also 
available in a corresponding way for $p=2q-1$, and that for $l\in\mathbb R$ with $|l|\le 1/q$
and $p\ge 2q-1$, our positive product formulas for the $\phi_\lambda^{p,l}$
lead to commutative hypergroup structures on $C_q$. This can be shown in the same way as in Section 4.
We here skip the details and remark only  that $\phi_{i\rho}^{p,l}\equiv 1$ for $\rho=\rho(k_{p,l})$ as in 
(\ref{def-rho-spezial})
 ensures 
that the corresponding positive measures on the right hand sides are in fact probability measures.

In summary:

\begin{theorem}\label{positive-prod}
 For all integers $q\ge1$, all $p\in [2q-1,\infty[$ and all 
 $l\in\mathbb R$ with $|l|\le 1/q$, the Heckman-Opdam hypergeometric functions 
$\phi_\lambda^{p,l}$ ($\lambda\in \mathbb C$) satisfy some  positive product formula  
(namely (\ref{prod-formel-allg-hypergeo})
 for $p> 2q-1$ and a corresponding one for  $p= 2q-1$).
Moreover, in this case the $\phi_\lambda^{p,l}$ ($\lambda\in \mathbb C$) 
are multiplicative functions of some associated unique commutative hypergroup structures
 $(C_q, *_{p.l})$.
\end{theorem}

It is not difficult to determine the Haar measures of these hypergroups:

\begin{proposition}\label{haar-quotient-hyp}
For integers $q\ge1$,  $p\in [2q-1,\infty[$ and  
 $l\in\mathbb R$ with $|l|\le 1/q$, the Haar measure on the hypergroup $(C_q, *_{p.l})$ is
    given by
\begin{equation}\label{haar-measure-chamber}
d\omega_{p,l}(t)=const\cdot
\prod_{j=1}^q \sinh^{2p-2q+1}t_j  \cosh^{2l+1} t_j\cdot \prod_{1\le i< j\le q}
\bigl| \cosh(2t_i)-\cosh(2t_j)\bigr|^2\> dt.
\end{equation}
\end{proposition}

We shall postpone the proof of this result to Section 6.

\begin{example} Consider $q=1$ and $p\ge 2q-1= 1$ as in Section \ref{q1--example-1}. 
We here have
$$\alpha:=k_1+k_2-1/2=p-1\ge0 ,
\quad\quad
 \beta:=k_2-1/2=l,$$
and $\phi_\lambda^{p,l}(t)=\phi_\lambda^{(\alpha,\beta)}(t)$ ($t\in[0,\infty[$)
for the Jacobi functions $\phi_\lambda^{(\alpha,\beta)}$ in \cite{K}.
We now use the parameters $\alpha\ge0$ and $\beta\in\mathbb R$ instead of $p,l$. 
Moreover, for $q=1$, we have $SU(1)=\{1\}$,
 $B_1=\{e^{i\theta}r: \> \theta\in [-\pi,\pi],\> r\in [0,1]\}$ and $dw=r\> dr\> d\theta$
 in the product formula (\ref{prod-formel-allg-hypergeo}). Therefore,
 (\ref{prod-formel-allg-hypergeo}) for $\alpha>0$, the symmetry of the integral
 w.r.t.~$\theta\in [-\pi,\pi]$ and the correct integration constant lead to the product formula
\begin{align}\label{prod-jacobi1}
\phi_\lambda^{(\alpha,\beta)}(s)\cdot \phi_\lambda^{(\alpha,\beta)}(t)=&
\frac{2\alpha}{\pi\cdot(\cosh s\cdot \cosh t)^\beta} 
\int_0^1\int_0^\pi 
 \phi_\lambda^{(\alpha,\beta)}({\rm arcosh}|r e^{i\theta}\sinh t\sinh s \,+\,\cosh t \cosh s|)\notag\\
&
\cdot {\rm Re}\Bigl( (r e^{i\theta}\sinh t\sinh s \,+\,\cosh t \cosh s)^\beta\Bigr) \cdot
(1-r^2)^{\alpha-1} r\> dr\> d\theta
\end{align}
for $s,t\ge0$, $\lambda\in\mathbb C$. For $\alpha=0$ we obtain the degenerate formula
\begin{align}\label{prod-jacobi2}
\phi_\lambda^{(0,\beta)}(s)\cdot \phi_\lambda^{(0,\beta)}(t)=&
\frac{1}{\pi(\cosh s\cdot \cosh t)^\beta} 
\int_0^\pi 
 \phi_\lambda^{(0,\beta)}({\rm arcosh}|r e^{i\theta}\sinh t\sinh s \,+\,\cosh t \cosh s|)\notag\\
&
\cdot {\rm Re}\Bigl( (r e^{i\theta}\sinh t\sinh s \,+\,\cosh t \cosh s)^\beta\Bigr) \> d\theta
\end{align}
For $\beta=0$, these formulas coincide with the well known product 
formulas for Jacobi functions;
 see Section 7 of \cite{K}. However, for $\beta\ne0$,  
(\ref{prod-jacobi1}) and (\ref{prod-jacobi2}) do not seem to be much used in literature. 
In fact, to our knowledge, they are only considered in Section 6 of \cite{RV}.
 It should be noticed that some details in Section 6 are not correct.

Let us compare our product formulas with those of Koornwinder \cite{K}.
Our Eqs.~(\ref{prod-jacobi1}) and (\ref{prod-jacobi2})
 are available for all  $\alpha\ge0$ and $\beta\in\mathbb R$, and they are
 positive for $\alpha\ge0$ and  $|\beta|\le1$.
On the other hand, Koornwinder's formulas in  Section 7 of \cite{K} are
available with positivity for $\alpha\ge\beta\ge -1/2$.
It is well known (see Section 7 of \cite{K},
and  \cite{J} and \cite{BH} for the hypergroup background) that  for all
$\alpha\ge\beta\ge -1/2$,
 Koornwinder's product formulas for the
 $\phi_\lambda^{(\alpha,\beta)}$ are associated with unique commutative
 hypergroup structures on $[0,\infty[$, the  so-called Jacobi hypergroups 
$([0,\infty[, *_{\alpha,\beta})$. As here injectivity of the Jacobi transform
   (as a special case of the injectivity of the  Fourier transform on
         commutative hypergroups in \cite{J}) ensures that
for $\alpha,\beta,s,t$ there is at most one bounded signed measure 
$\mu_{s,t}^{\alpha,\beta}\in M_b([0,\infty[)$ with
$$ \phi_\lambda^{(0,\beta)}(s)\phi_\lambda^{(0,\beta)}(t)=\int\phi_\lambda^{(0,\beta)}(u)
\> d\mu_{s,t}^{\alpha,\beta}(u)$$
for all $\lambda\in\mathbb C$, we conclude that for $\alpha\ge\max(\beta,0)$
and $\beta\ge -1/2$, the 
product formulas (\ref{prod-jacobi1}) and (\ref{prod-jacobi2})
 are equivalent to those of  \cite{K}
 and thus positive.

 This in particular shows that for $q=1$ the range of parameters $p,l$ in Theorem
\ref{positive-prod} with a positive product formula is larger than described there.
 We expect that this holds also for $q\ge2$.

Taking our results and the results of Koornwinder into account, we obtain in
summary that the Jacobi functions  $\phi_\lambda^{(0,\beta)}$ ($\lambda\in
\mathbb C$) admit associated commutative hypergroup structures for the set of
parameters
$$\{(\alpha,\beta)\in\mathbb R^2:\> \alpha\ge\beta\ge -1/2 \quad\text{or}\quad
\alpha\ge0, \> \beta\in[-1,1]\}.$$
By  classical results in the case of Koornwinder (see \cite{BH} for details on
the Jacobi hypergroups) and by
(\ref{haar-measure-chamber}) in our case,
the Haar measures on these hypergroups are given in both cases by
$$const \cdot \sinh^{2\alpha+1}t \cdot \cosh^{2\beta+1}t \quad\quad(t\ge0).$$
\end{example} 

\section{Signed hypergroups on $C_q$}

We show in this section that for all $p\ge 2q-1$ and $l\in\mathbb R$, 
the (not necessarily positive) product formulas of Section 5 for the Heckman-Opdam functions
 $\phi_\lambda^{p,l}$ on $C_q$ are related to some so-called signed hypergroup 
structure $(C_q, \bullet_{p,l})$.
 For this we return to the commutative hypergroups $(C_q\times\mathbb R, *_p)$
of Section 4 for $p\ge 2q-1$ which have the functions $\psi_{\lambda,l}^p$ as multiplicative functions.
As in Section 4, we consider the closed subgroup $G:=\{0\}\times \mathbb R$ of  $(C_q\times\mathbb R, *_p)$,
and identify the quotient $(C_q\times\mathbb R)/G$ with $C_q$. 

For   $l\in\mathbb R$ consider the functions $\sigma_l\in C_b(C_q\times\mathbb R)$
with $\sigma_l(t,\theta):=e^{il\theta}$. They satisfy
$$|\sigma_l(t,\theta)|=1, \quad \sigma_l(t,-\theta)= \overline{\sigma_l(t,\theta)}, 
\quad\text{and}\quad
(\delta_{(0,\tau)}*_p\delta_{(t,\theta)})( \sigma_l)=\sigma_l(t,\theta)\cdot \sigma_l(0,\tau)$$
for $t\in C_q$ and $\theta,\tau\in\mathbb R$, i.e., the
$\sigma_l$ are partial characters on  $(C_q\times\mathbb R, *_p)$ w.r.rt.~$G$ 
in the sense of Definition 4.1 of \cite{RV}. For  $l\in\mathbb R$ we now consider the
 mapping
\begin{align}
 (C_q\times\mathbb R)/G \> \times (C_q\times\mathbb R)/G \> \longrightarrow& \>
M_b(C_q\times\mathbb R)/G ) \notag\\
((s,\theta_1)*G, (t,\theta_2)*G)\> \longmapsto &\> \delta_{(s,\theta_1)*G}
\bullet_{p,l} \delta_{(t,\theta_2)*G}:=
 \notag\\
&:=\overline{ \sigma_l(s,\theta_1)}\cdot \overline{ \sigma_l(t,\theta_2)}\cdot pr(\sigma_l\cdot(
\delta_{(s,\theta_1}*_p\delta_{(t,\theta_2)}))
\notag\end{align}
with $pr$ as the canonical projection $pr:C_q\times\mathbb R\to (C_q\times\mathbb R)/G$ 
as well as  its extension to images of measures. It can be easily checked
 (see also Section 4   of \cite{RV} for a general theory)
that this mapping  is well-defined, i.e.,
 independent of representatives of the cosets, and weakly continuous. Moreover, by  
 Section 4   of \cite{RV}, it can be uniquely extended in a bilinear, weakly continuous way
to an associative convolution $*_l$ on $M_b((C_q\times\mathbb R)/G )$. Moreover, 
$(M_b((C_q\times\mathbb R)/G ), *_l)$ is a commutative Banach-$*$-algebra with respect
 to the total variation norm with the identity $\delta_{(0,0)*G}$ as neutral element where
the involution on $ M_b((C_q\times\mathbb R)/G)$ is inherited from that on
the hypergroup $C_q\times\mathbb R$. We also consider the
the  image
$\tilde\omega_p:=pr(\omega_p)\in M^+((C_q\times\mathbb R)/G)$ 
of the Haar measure $\omega_p$ on  $C_q\times\mathbb R$ in (\ref{haar-measure-full}).
By Theorem 4.6 of \cite{RV}, the quotient $(C_q\times\mathbb R)/G$
with the convolution  $\bullet_{p,l}$ and the measure $\tilde\omega_p$ forms a
so-called commutative signed hypergroup
 $((C_q\times\mathbb R)/G,\bullet_{p,l},\tilde\omega_p)$ with the identity
mapping as involution, i.e., a so-called hermitian commutative signed hypergroup. 
For details on this notion we refer to \cite{R0}, \cite{RV}
and references cited there. We only remark that for a Haar  measure $\omega_p$
in this setting 
we require the conjugation relation
\begin{equation}\label{conjugation}
\int_{(C_q\times\mathbb R)/G} T_xf \cdot g \> d\omega_p=
\int_{(C_q\times\mathbb R)/G} T_xg \cdot f \> d\omega_p
\end{equation}
for all $f,g\in C_c((C_q\times\mathbb R)/G)$ and $x\in(C_q\times\mathbb R)/G$
 where the translates $T_x$ are  defined by $T_xf(y):=(\delta_x\bullet_{p,l}\delta_y)(f)$.
It is well-known (see \cite{J}) that for usual hermitian commutative
hypergroups, the usual Haar measures satisfy  (\ref{conjugation}), and that by
\cite{RV},  Haar  measures on  commutative signed hypergroups in the above
sense are unique up to
a constant.

We now identify $(C_q\times\mathbb R)/G$ with $C_q$ as usual.  In this case, 
 we obtain from (\ref{gen-twodim-convolution}) that for $p>2q-1$
the convolution $\bullet_{p,l}$ on $C_q$ is given by
\begin{align}(\delta_s\bullet_{p,l}\delta_t)(f)&= \int_{C_q\times\mathbb R}
  e^{il\theta}\> d(\delta_{s,0}*_p \delta_{s,0})(t,\theta)\\
&=\frac{1}{\kappa_p}\int_{B_q}\int_{SU(q)} f(d(s,t;v,w))\cdot(arg\> h(s,t;v,w))^l
\cdot
\Delta(I_q-w^*w)^{p-2q}
\> dv\> dw.\notag
\end{align}
Moreover, for $p=2q-1$ we obtain a corresponding formula from
(\ref{convo-degen-2}), and the Haar measure $\omega_p\in M^+(C_q)$ of our
signed hypergroups is independent of $l$ and is given by
\begin{equation}\label{haar-measure-quotient}
d\tilde\omega_p(t)=const\cdot\prod_{j=1}^q \sinh^{2p-2q+1}t_j 
 \cosh t_j\cdot \prod_{1\le i< j\le q}
\bigl| \cosh(2t_i)-\cosh(2t_j)\bigr|^2\> dt.
\end{equation}
Furthermore, using the multiplicative functions $\psi_{\lambda,l}^p$ on the
hypergroups $(C_q\times\mathbb R, *_p)$, we obtain immediately that for
$l\in\mathbb R$,
the functions
$$\tilde\phi_{\lambda}^{p,l}(t):=\prod_{j=1}^q \cosh^l t_j\cdot
F_{BC_q}(i\lambda,k(p,q,l);t) \quad\quad(\lambda\in\mathbb C)$$
are multiplicative functions of our signed quotient hypergroups.

It should be noticed that except for $l=0$, i.e., the case of a classical
quotient hypergroup, our signed quotient hypergroups together their 
 multiplicative functions $\tilde\phi_{\lambda}^{p,l}$ are different, but
 closely related to the convolution of measures on $C_q$ which is induced by
 the product formulas for the functions $\phi_{\lambda}^{p,l}$ in Section 5.

 In fact, we have 
$$\tilde\phi_{\lambda}^{p,l}(t)=\prod_{j=1}^q \cosh^l
 t_j\cdot\phi_{\lambda}^{p,l}(t)$$
and, for $s,t\in C_q $,
$$\delta_s\bullet_{p,l}\delta_t=
\frac{f_0(s)\cdot f_0(t)}{f_0} (\delta_s*_{p,l}\delta_t)$$
 with the positive function
$$f_0(t):=\prod_{j=1}^q \cosh^l t_j=\tilde\phi_{i\rho}^{p,l}(t).$$
 By the following lemma, the Haar measures $\tilde\omega_p$ 
associated with  $\bullet_{p,l}$ are related with the 
Haar measures $\omega_{p,l}$  of the hypergroups
$(C_q,*_{p,l})$ of Section 5 for $|l|\le 1/q$.
This observation together with (\ref{haar-measure-quotient}) then lead
immediately to the Haar measures in
Proposition \ref{haar-quotient-hyp}.

\begin{lemma} Let $X$ be a locally compact Hausdorff space and $h\in C(X)$ a
  positive continuous function on  $X$.
Assume we have two signed hermitian commutative hypergroup structures
$(X, *,\omega_*)$ and $(X,\bullet ,\omega_\bullet)$ on $X$ with the
convolutions $*,\bullet$ and associated Haar measures $\omega_*
,\omega_\bullet$. Assume that the convolutions are related by
$$\delta_x\bullet\delta_y =\frac{h}{h(x)h(y)}\cdot (\delta_x*\delta_y )
\quad\quad \text{for}\quad x,y\in X,$$
then, up to a positive multiplicative constant, the  Haar measures are related
by
$\omega_\bullet=h^2\omega_*$. Moreover, a function $f\in C(X)$ is
multiplicative w.r.t.~the convolution $*$ 
if and only on $f/h$ is multiplicative w.r.t.~$\bullet$.
\end{lemma}

\begin{proof} Let $\omega_*$ be a Haar measure associated with $*$, and denote
  the translates of $f\in C_c(X)$ by $x\in X$ w.r.t.~ $*$ and $\bullet$ by
  $T^*_xf$ and  $T^\bullet_xf$ respectively.
 Then, by the conjugation relation for  $\omega_*$, we
  have for $f,g\in C_c(X)$, $x\in X$ and $\omega_\bullet:=h^2\omega_*$ that
\begin{align}
\int_X T^\bullet_xf\cdot g\> d\omega_\bullet&=
\int_X   \frac{T^*_x(hf)(y)}{h(x)h(y)} \cdot g(y) h^2(y)\> d\omega_*=
\frac{1}{h(x)} \int_X   T^*_x(hf)(y) \cdot g(y) h(y)\> d\omega_*=\notag\\
&=\frac{1}{h(x)} \int_X h(y)f(y)\cdot T^*_x(hg)(y)\> d\omega_*= ...=
\int_X T^\bullet_xg\cdot f\> d\omega_\bullet.\notag
\end{align}
This yields the conjugation relation for $\omega_\bullet$. As Haar measures on
signed hypergroups are unique up to a constant by \cite{RV}, the first part of the lemma
follows. The second part of the lemma is also clear.
\end{proof}

Let us summarize the preceding observations in Sections 5 and 6.
 In principle, for $l\in\mathbb R$
and $p\ge 2q-1$, we have two  choices to define convolution
structures 
for measures on $C_q$ associated with the hypergeometric functions $F_{BC_q}$,
namely the convolutions of Section 5 for the functions  $F_{BC_q}$ directly as
well as the convolutions  of Section 6 for the functions 
$\tilde\phi_\lambda^{p,l}$. These both convolutions are equal for $l=0$ in
which case $C_q$ carries the usual classical quotient convolution which was 
studied in \cite{R2}.

 For $l\ne0$ with  $l\in[-1/q,1/q]$, we have a positive
convolution. In this case, the convolution  of Section 5 for
 the functions  $F_{BC_q}$ is probability preserving and generates 
classical hypergroup structures, which is not the case for the convolutions
of  Section 6, which are only positive and norm-decreasing. We thus prefer 
the point of view of Section 5 in this case.

On the other hand, for $l\in\mathbb R\setminus[-1/q,1/q]$, 
it is unclear (at least for $q\ge2$) in which cases our convolutions are positive. 
It is also unclear in this case whether the convolutions of Section 5 are
norm-decreasing (w.r.t.~total variation norm), while this is the case always 
for the convolutions of Section 6 by their very construction above.
 We thus prefer the convolution of Section 6 in those cases,
 where no positivity is known.

\end{document}